\documentclass{article}

\usepackage{amsmath}
\usepackage{amssymb}
\usepackage{graphicx}
\usepackage{bm}
\usepackage{mathdots}
\usepackage{booktabs}
\usepackage{rotating}
\usepackage{mathtools}

\usepackage{url}
\usepackage{subfigure}

\usepackage[ruled,linesnumbered]{algorithm2e}
\usepackage[a4paper,left=2.8cm,right=2.8cm,top=2.5cm,bottom=2.2cm]{geometry}
\usepackage{fancyhdr}
\usepackage[framemethod=tikz]{mdframed}
\newmdtheoremenv[%
backgroundcolor=green!10,%
outerlinecolor=black,%
leftmargin=0,%
rightmargin=0,
innertopmargin =3pt,%
innerleftmargin = 5pt,
innerrightmargin = 5pt,
splittopskip = \topskip,%
skipabove = \baselineskip,%
skipbelow = \baselineskip,%
roundcorner=5, ntheorem]
{theorem}{Theorem}[section]
\newtheorem{corollary}{Corollary}[section]

\newtheorem{lemma}{Lemma}[section]

\makeatletter 
\@addtoreset{equation}{section}
\makeatother  

\newtheorem{remark}{Remark}[section]
\newtheorem{assumption}{Assumption}[section]

\newenvironment{proof}{{\noindent\it Proof.}\quad}{\hfill $\square$\\}

\newcommand{\Lt}{L^2}
\newcommand{\Hs}{H^s}

\begin{document}
\title{Breaking quadrature exactness: A spectral method for the Allen--Cahn equation on spheres}

\author{Hao-Ning Wu\footnotemark[1]
       \quad\quad Xiaoming Yuan\footnotemark[2]}

\footnotetext[1]{Department of Mathematics, University of Georgia, Athens, GA (hnwu@uga.edu).}
\footnotetext[2]{Department of Mathematics, The University of Hong Kong, Hong Kong, China (xmyuan@hku.hk).}


\maketitle

\begin{abstract}
We present a novel spectral method for the Allen-Cahn equation on spheres, eliminating the reliance on conventional quadrature exactness conditions. By replacing these conditions with a restricted isometry relation derived from Marcinkiewicz--Zygmund quadrature systems, our method achieves precise control over quadrature errors for polynomial integrands. This theoretical advancement enables the use of substantially more choices of quadrature points than classical spectral methods while maintaining rigorous error bounds. The proposed method requires only mild constraints on the polynomial degree of numerical solutions to establish both the maximum principle and energy stability, representing a considerable departure from existing techniques that depend on restrictive time stepping sizes,  Lipschitz property of the nonlinear term, or $L^{\infty}$ boundedness of numerical solutions. Notably, our method permits time stepping sizes independent of the diffusion coefficient, making it suitable for long-time simulations. Inspired by the effective maximum principle proposed by Li (Ann. Appl. Math., 37(2): 131--290, 2021), we develop an almost sharp maximum principle that allows controllable deviation of numerical solutions from the sharp bound with large time stepping sizes. Furthermore, we prove that when the quadrature rule attains sufficient exactness, our method preserves energy stability and coincides mathematically with the Galerkin method. In addition, we propose an energy-stable mixed-quadrature scheme which works well even with randomly sampled initial condition data. Our numerical experiments on $\mathbb{S}^2$ validate the theoretical results about the energy stability and the almost sharp maximum principle.
\end{abstract}

\textbf{Keywords: }{spectral method, hyperinterpolation, maximum principle, energy stability, \\
\indent \indent \indent~~~~\indent quadrature exactness, Allen--Cahn, Marcinkiewicz--Zygmund}

\smallskip


\bigskip

\section{Introduction}

Partial differential equations (PDEs) posed on the sphere play a central role in modeling geological, meteorological, and oceanic processes, with the sphere serving as an idealized representation of Earth. Moreover, solving PDEs on the sphere represents the simplest case of the broader challenge of solving PDEs on smooth compact manifolds.

In this paper, we propose a novel spectral method for numerically solving stiff and semi-linear PDEs on the unit sphere $\mathbb{S}^{d-1}=\{x\in\mathbb{R}^d:\|x\|_2=1\}\subset\mathbb{R}^d$ with dimension $d\geq 3$. We consider equations of the form
\begin{equation}\label{equ:PDE}
u_t=\mathbf{L}u+\mathbf{N}(u),\quad u(0,x) = u_0(x),
\end{equation}
where $u = u(t, x)$ with $(t,x)\in[0,\infty)\times\mathbb{S}^{d-1}$ is a function of time $t$ and spatial variable $x\in \mathbb{S}^{d-1}$, $\mathbf{L}$ is a constant-coefficient linear differential operator, and $\mathbf{N}$ is a constant-coefficient nonlinear differential (or non-differential) operator of lower order. Although our primary focus is on the sphere, the proposed numerical method can be readily extended to other compact manifolds and even Euclidean domains, provided an orthogonal polynomial basis is available. This adaptability broadens the method's applicability beyond the sphere.

To demonstrate our method, we consider the Allen--Cahn equation
\begin{equation}\label{equ:AC}
u_t=\nu^2\Delta u-F'(u), \quad u(0,x) = u_0(x)
\end{equation}
on the sphere $\mathbb{S}^{d-1}$ as a model equation, where $\Delta$ is the Laplace--Beltrami operator on $\mathbb{S}^{d-1}$. Introduced in \cite{allen1979microscopic} for describing the process of phase separation in iron alloys, the Allen--Cahn equation \eqref{equ:AC} is a reaction-diffusion equation with a linear diffusion term $\nu^2\Delta u$ and a nonlinear reaction term $F'(u)$. The solution $u=u(t,x)$ is a scalar function typically representing the concentration of one of the two metallic components of the alloy. The nonlinear term has the usual double well form of $F'(u)=f(u) = u^3-u$ with $F(u)=\frac14(u^2-1)^2$. This equation (\ref{equ:AC}) possesses two intrinsic properties: energy stability and the maximum principle. We consider the energy functional
\begin{equation}\label{equ:energy}
\mathcal{E}(u):=\int_{\mathbb{S}^{d-1}}\left(\frac{\nu^2}{2}|\nabla u|^2+F(u)\right)\text{d}\omega_d,
\end{equation}
where $\text{d}\omega_d$ is the surface measure on $\mathbb{S}^{d-1}$. That is, $\int_{\mathbb{S}^{d-1}}\text{d}\omega_d = |\mathbb{S}^{d-1}|$ denotes the surface area of $\mathbb{S}^{d-1}$. Note that the $L^2$-gradient flow $\mathcal{E}$ is a decreasing function of the time $t$ in the sense of
\begin{equation*}
\frac{\text{d}\mathcal{E}(u(t))}{\text{d}t} = -\int_{\mathbb{S}^{d-1}} |\nu^2\Delta u -F'(u)|^2\text{d}\omega(x) \leq0.
\end{equation*}
Therefore, for smooth solutions of the equation \eqref{equ:AC}, it holds that the energy decay $\mathcal{E}(u(t,\cdot))\leq \mathcal{E}(u(s,\cdot))$ for any $0\leq s\leq t <\infty$. Moreover, due to the particular structure of the Allen--Cahn equation \eqref{equ:AC}, we also have the $L^{\infty}$ maximum principle for the solution to \eqref{equ:AC}. That is, if the $L^{\infty}$ norm of $u_0$ is bounded by some constant, then that of the entire solution should also be bounded by the same constant.

\subsection{Motivations}

In this paper, we aim to propose a quadrature-based spectral method for the Allen--Cahn equation \eqref{equ:AC} on $\mathbb{S}^{d-1}$ with all numerical solutions being spherical polynomials of degree $N$. Our three-fold motivations arise from the practical simulation of the Allen--Cahn equation \eqref{equ:AC}.

\noindent \textbf{Motivation I: Time stepping size and stringent conditions.} For the Allen--Cahn equation \eqref{equ:AC} and many related phase-field models, various numerical methods have been proposed to preserve the energy stability and the sharp maximum principle. For the literature on ensuring the energy stability (or modified ones) and preserving the maximum principle in the numerical simulation of the Allen--Cahn equation \eqref{equ:AC} and related phase field models, we refer to, e.g., \cite{MR1742748,blowey1996numerical,bonnetier2012consistency,chen1998applications,MR4378430,du,MR3049920,MR3407238,feng2003numerical,MR2101784,li2022large,shen2010numerical,MR3719623,zhang2009numerical} and references therein. Although preserving both properties is highly desirable for numerical simulations, sometimes only modified energy stability can be analyzed, and some unwanted stringent conditions on the numerical schemes are always introduced. For example, in the stiff case of the diffusion coefficient $\nu \ll 1$, standard numerical methods may require extremely small time stepping sizes depending on $\nu$ to maintain stability. Moreover, numerical methods for \eqref{equ:AC} with unconditional stability for any time stepping size have been studied in the literature, e.g., \cite{bertozzi2010biharmonic,MR2815678,shen2012second}, while these methods usually rely on strict assumptions such as the derivative of the nonlinear term is Lipschitz, or the numerical solutions $u^n$ have \emph{a priori} $L^{\infty}$ bounds.

Therefore, our first motivation is to remove all these stringent and technical conditions and to develop stable numerical methods allowing larger time steps, as long-time simulations are often necessary for phase-field models. We impose some conditions onto the degree $N$ only, and the degree $N$ is independent of the time stepping size. This idea is motivated by the \emph{effective maximum principle} proposed recently by Li in \cite{MR4294331}, which is an almost sharp maximum principle and allows the numerical solutions to deviate from the sharp bound of solutions by a controllable discretization error. Without the aforementioned strict assumptions on the derivative of the nonlinear term or \emph{a priori} $L^{\infty}$ bounds of the numerical solutions, the effective maximum principle in \cite{MR4294331} is more favorable for numerical analysis of spectral methods for solving semi-linear PDEs. It is worth noting that for finite difference schemes, discrete energy stability was shown in \cite{MR3549191} to hold for $0 < \tau \leq 1/2$. This result was later extended to spectral methods for the Allen--Cahn equation \eqref{equ:AC} on the torus, where stability was proved for $0 < \tau \leq 0.86$ \cite{MR4294331}.

\noindent \textbf{Motivation II: Numerical integration and discrete orthogonal projection.} Our second motivation is concerned with the sampling process in fully discrete practical simulation, which involves the usage of orthogonal projection operators and numerical integration. The $L^2$ orthogonal projection plays an essential role in spectral Galerkin methods. On the sphere $\mathbb{S}^{d-1}$, a convenient $L^2$-orthonormal basis (with respect to $\text{d}\omega_d$) for the space $\mathbb{P}_N:=\mathbb{P}_N(\mathbb{S}^{d-1})$ of polynomials of degree at most $N$ is provided by spherical harmonics $\{Y_{\ell,k}:k=1,2,\ldots Z(d,\ell);\ell=0,1,2,\ldots,N\}$ with dimension $\dim \mathbb{P}_N = Z(d+1,N)$, where
\begin{equation*}
Z(d,0) = 1, ~ Z(d,\ell) = (2\ell +d-2)\frac{\Gamma(\ell+d-2)}{\Gamma(d-1)\Gamma(\ell+1)}\sim \frac{2}{\Gamma(d-1)}\ell^{d-2}\text{ as }\ell\rightarrow\infty;
\end{equation*}
see, e.g., \cite{MR2934227,MR0199449}. The orthogonal projection on $\mathbb{S}^{d-1}$ of $f\in L^2(\mathbb{S}^{d-1})$ is then defined as
\begin{equation}\label{equ:L2projection}
\mathcal{P}_Nf=\sum_{\ell=0}^{N}\sum_{k=1}^{Z(d,\ell)}\langle f,Y_{\ell,k}\rangle Y_{\ell,k},
\end{equation}
with the inner product defined as
\begin{equation}\label{equ:innerproduct}
\langle v,z\rangle:=\int_{\mathbb{S}^{d-1}}vz\text{d}\omega_d.
\end{equation}

In \cite{MR4294331}, the following implicit-explicit spectral scheme was proposed for the Allen--Cahn equation \eqref{equ:AC} on the torus $[0,1)^d$ ($d=1,2,3$) with periodic boundary conditions:
\begin{equation}\label{equ:prototypicalscheme}
\begin{cases}
&\dfrac{u^{n+1}-u^n}{\tau}=\nu^2\Delta u^{n+1}-\mathcal{P}_{N}\left((u^n)^3-u^n\right),\quad n\geq 0,\vspace{0.15cm}\\
& u^0 = \mathcal{P}_{N}u_0,
\end{cases}
\end{equation}
where $\tau>0$ is the size of time step, $u^n$ denotes the numerical solution at time $t=n\tau$, and the operator $\mathcal{P}_N$ projects any periodic function to its first $N$ Fourier modes. For the Allen--Cahn equation \eqref{equ:AC} on other non-periodic domains, the numerical scheme \eqref{equ:prototypicalscheme} can be similarly implemented using the $L^2$ orthogonal projection operator $\mathcal{P}_N$ mapping $L^2$ functions to $\mathbb{P}_N$. The spherical case presented in \eqref{equ:L2projection} serves as a concrete example of the projection operator on $\mathbb{S}^{d-1}$. The evolution scheme \eqref{equ:prototypicalscheme} is equivalent to the \emph{Galerkin} scheme
\begin{equation}\label{equ:ACGalerkin}
\left\langle \frac{u^{n+1}-u^n}{\tau},\chi \right\rangle = \left\langle\nu^2\Delta u^{n+1},\chi \right\rangle - \left\langle (u^n)^3-u^n, \chi  \right\rangle \quad \forall \chi \in \mathbb{P}_N;
\end{equation}
this equivalence can be immediately shown with the projection property $\mathcal{P}_N\chi=\chi$ for all $\chi\in \mathbb{P}_N$.

For practical simulations via the scheme \eqref{equ:prototypicalscheme}, however, the inner products in either the Galerkin scheme \eqref{equ:ACGalerkin} or the orthogonal projection operator \eqref{equ:L2projection} occurred in the scheme \eqref{equ:prototypicalscheme} should be evaluated by some quadrature rules. For example, an $m$-point positive-weight spherical quadrature rule takes the form of
\begin{equation}\label{equ:quad}
\sum_{j=1}^mw_jg(x_j)\approx \int_{\mathbb{S}^{d-1}}g\text{d}\omega_d,
\end{equation}
where quadrature points $x_j\in\mathbb{S}^{d-1}$ and weights $w_j>0$ for $j=1,2,\ldots,m$. For numerical integration on the sphere, we refer the reader to \cite{MR2263736,hesse2010numerical,MR2065291}. Therefore, we are motivated to incorporate the effects of numerical integration into our analysis of the scheme for solving the Allen--Cahn equation \eqref{equ:AC} on $\mathbb{S}^{d-1}$. This consideration is crucial, as an analysis of \eqref{equ:prototypicalscheme} alone may not fully capture the true behavior of the numerical solutions.

\noindent \textbf{Motivation III: Limited samples.} Conventionally, the quadrature rules are always chosen to have the exactness degree of $2N$ for numerical solutions on $\mathbb{P}_N$. That is, we have
\begin{equation}\label{equ:exactness}
\sum_{j=1}^mw_jg(x_j) = \int_{\mathbb{S}^{d-1}}g\text{d}\omega_d\quad\forall g\in\mathbb{P}_{2N}.
\end{equation}
For instance, in the numerical treatment of PDEs on one-dimensional domains, spectral methods commonly employ Gauss quadrature rules. This approach requires function evaluations at predetermined locations, known as the quadrature points, which are inherent to well-established quadrature rules.

Our third motivation then arises from the following question: What if we do not have full access to the initial data $u_0$ but only a set of samples $\{u_0(x_j)\}_{j=1}^m$, whose data locations $\{x_j\}_{j=1}^m$ cannot be prescribed? In this case, the quadrature rule \eqref{equ:quad} with points $\{x_j\}_{j=1}^m$ might not have the desired exactness \eqref{equ:exactness}, but it is still necessary to investigate behaviors of the numerical solutions obtained from these limited samples. On the one hand, such a consideration comes in line with the trend of interest in the numerical analysis community that the necessity of quadrature exactness should be re-accessed. This is because what matters in practice is the accuracy for integrating non-polynomial functions, see, e.g., \cite{an2022quadrature,trefethen2022exactness}. 

On the other hand, even when the quadrature rule \eqref{equ:quad} with exactness can be used, this investigation is still necessary. Numerical integration on surfaces differs fundamentally from integration in Euclidean spaces. While Gauss quadrature rules provide exact integration for polynomials on intervals (e.g., $[-1,1]$), and tensor product constructions extend this to higher-dimensional Euclidean domains, such approaches may not be directly applied to surfaces like the sphere due to geometric constraints. We take the sphere as an example to explain the challenge. A spherical $t$-design, introduced in \cite{delsarte1991geometriae}, is a set of points $\{x_j\}_{j=1}^m\subset \mathbb{S}^{d-1}$ with the characterization that an equal-weight quadrature rule in these points exactly integrates all polynomials of degree at most $t$, that is,
\begin{equation*}
\frac{|\mathbb{S}^{d-1}|}{m}\sum_{j=1}^m\chi(x_j)=\int_{\mathbb{S}^{d-1}}\chi(x)\text{d}\omega_d(x)\quad\forall \chi\in\mathbb{P}_t.
\end{equation*}
Therefore, the quadrature rule \eqref{equ:quad} with quadrature points as a spherical $2N$-design satisfies the quadrature exactness requirement \eqref{equ:exactness}. It was verified in \cite{MR3071504} that, for each $m\geq ct^{d-1}$ with some positive but unknown constant $c>0$, there exists a spherical $t$-design in $\mathbb{S}^{d-1}$ consisting of $m$ points. However, the distribution of a specific spherical $t$-design is still unknown. In practice, spherical $t$-designs are obtained by solving equivalent optimization problems, see, e.g., \cite{MR2763659,MR2272596,MR3822282}. While precomputed designs are available for moderate values of $t$, the computational cost becomes prohibitive for large $t$ due to the complexity of the underlying optimization. This limitation also motivates our approach of using quadrature rules with lower-degree exactness when high-degree designs are unavailable.

\subsection{New spectral scheme}

Consider discretizing the orthogonal projection operator $\mathcal{P}_N$ directly in the scheme \eqref{equ:prototypicalscheme} as
\begin{equation}\label{equ:hyperinterpolation}
\mathcal{L}_{N}f=\sum_{\ell=0}^{N}\sum_{k=1}^{Z(d,\ell)}\langle f,Y_{\ell,k}\rangle_m Y_{\ell,k},
\end{equation}
where
\begin{equation}\label{equ:discreteinnerproduct}
\langle v,z\rangle_m:=\sum_{j=1}^mw_jv(x_j)z(x_j)
\end{equation}
is a ``discrete version'' of the $L^2$ inner product \eqref{equ:innerproduct}. This is a fully discrete scheme. Note that the operator \eqref{equ:hyperinterpolation} is now always referred to as the hyperinterpolation operator, which was originally introduced by Sloan in \cite{sloan1995polynomial}. Hence, for the Allen--Cahn equation \eqref{equ:AC} on the sphere $\mathbb{S}^{d-1}$, we propose the following spectral scheme:
\begin{equation}\label{equ:scheme}
\begin{cases}
&\dfrac{u^{n+1}-u^n}{\tau}=\nu^2\Delta u^{n+1}-\mathcal{L}_{N}\left((u^n)^3-u^n\right),\quad n\geq 0,\vspace{0.15cm}\\
& u^0 = \mathcal{L}_{N}u_0.
\end{cases}
\end{equation}
It should also be mentioned that spherical harmonics are eigenfunctions of the negative Laplace--Beltrami operator on the sphere. Thus we can avoid the discretization of spatial differential operators, and the scheme \eqref{equ:scheme} is already fully discrete. Moreover, the implementation of the scheme \eqref{equ:scheme} needs to update the coefficients of the numerical solution $u^n$, and it requires only vector-matrix multiplications. During each time evolution from $n$ to $n+1$, we need to evaluate the coefficients of the hyperinterpolant $\mathcal{L}_N((u^n)^3-(u^n))$, which can be accomplished in $\dim{\mathbb{P}_N}+2m$ floating point operations (flops). We also need to update the coefficients of $u^{n+1}$, which can be done in $3\dim{\mathbb{P}_N}$ flops. Therefore, each time evolution of the scheme \eqref{equ:scheme} can be achieved in $2\left((m+1)(\dim{\mathbb{P}_N})+m\right)$ flops, and thus it has same theoretical benefits of the Galerkin method \eqref{equ:ACGalerkin} at a computational cost comparable to the collocation method.

\begin{remark}
The numerical scheme \eqref{equ:scheme} naturally extends to semi-linear PDEs \eqref{equ:PDE} on general non-spherical domains through proper definition of the hyperinterpolation operator. Such definition fundamentally requires: (i) an orthonormal basis for $\mathbb{P}_N$, and (ii) a quadrature rule. Unlike implementations on tori or spheres, this generalized setting may additionally require computation of basis polynomial derivatives for differential operator discretization.
\end{remark}

In \cite{sloan1995polynomial}, the construction of hyperinterpolation relies on the quadrature exactness \eqref{equ:exactness}. However, recent works in \cite{an2022bypassing,an2022quadrature} have relaxed and even bypassed this assumption. In this paper, the quadrature exactness \eqref{equ:exactness} is not a necessary assumption for our scheme; we only make the following three natural and simple assumptions:
\begin{assumption}\label{assumption:1MZ}
For the quadrature rule \eqref{equ:quad}, we assume that
\begin{enumerate}
\item[{\rm{(I)}}] it integrates all constants exactly; namely, $\sum_{j=1}^mw_j= \int_{\mathbb{S}^{d-1}}{\rm{d}}\omega_d = |\mathbb{S}^{d-1}|$;
\item[{\rm{(II)}}] $\{(x_j,w_j)\}_{j=1}^m$ forms a Marcinkiewicz--Zygmund (MZ) system of order 2 with respect to $\mathbb{P}_{N}$; namely, for every $N\geq 0$ and $\chi\in\mathbb{P}_{N}$, there exists a constant $\eta<1$, independent of $\chi$ and $N$, such that
 \begin{equation}\label{equ:MZproperty}
\left|\sum_{j=1}^mw_j\chi(x_j)^2-\int_{\mathbb{S}^{d-1}}\chi^2\text{d}\omega_d\right|\leq \eta \int_{\mathbb{S}^{d-1}}\chi^2\text{d}\omega_d\quad \forall \chi\in\mathbb{P}_{N};
\end{equation}
\item[$\rm{(III)}$] it converges to $\int_{\mathbb{S}^{d-1}}g{\rm{d}}\omega_d$ as $m\rightarrow \infty$ for all $g\in C(\mathbb{S}^{d-1})$.
\end{enumerate}
\end{assumption}

Assumption (I) holds if either the quadrature rule \eqref{equ:quad} is equal-weight, i.e., $w_j=|\mathbb{S}^{d-1}|/m$ for all $j=1,2,\ldots,m$, or the quadrature rule \eqref{equ:quad} has exactness degree at least one. If this assumption does not hold, we only need to replace the term $|\mathbb{S}^{d-1}|$ in our theoretical results with $\sum_{j=1}^mw_j$. Assumption (II) is equivalent to the MZ inequality, which has been intensively investigated in \cite{filbir2011marcinkiewicz,mhaskar2001spherical}. From a numerical perspective, Assumption (II) merely indicates that the relative error of evaluating the integral of $\chi^2$ via the rule \eqref{equ:quad} should be less than one for any $\chi\in\mathbb{P}_N$. Moreover, it should be noted that Assumption (II) implies $m\rightarrow\infty$, as $N\rightarrow\infty$. Assumption (III) is a natural assumption regarding the performance of quadrature rules.

\subsection{Outline}

In the paper, our main purpose is to investigate the $L^{\infty}$ stability and energy stability for the new scheme \eqref{equ:scheme}, as well as establish the effective maximum principle. In the next section, we introduce some preliminaries on spherical harmonics and the Sobolev space on spheres. In Section \ref{sec:theory}, for the scheme \eqref{equ:scheme} with quadrature rules \eqref{equ:quad} only fulfilling Assumption \ref{assumption:1MZ}, we establish the $L^{\infty}$ stability for $0<\tau<2$ and the effective maximum principle for $0<\tau\leq 1/2$. In Section \ref{sec:discussion}, we demonstrate that if the quadrature exactness \eqref{equ:exactness} is assumed, then the new scheme \eqref{equ:scheme} is equivalent to a fully discrete Galerkin method and it has discrete energy stability for $0<\tau\leq 0.86$. Moreover, if the quadrature rule \eqref{equ:quad} is assumed to have exactness degree of $4N$, we demonstrate the stability of the original energy \eqref{equ:energy}. Our theoretical assertions are verified by some numerical experiments on the unit sphere $\mathbb{S}^2$ in Section \ref{sec:example}. In Section \ref{sec:conclusion}, some conclusions are drawn and some discussions are initiated for future research.

\section{Preliminaries}

We are concerned with real-valued functions on the sphere $\mathbb{S}^{d-1}$ in the Euclidean space $\mathbb{R}^{d}$ for $d\geq 3$. For the case of $d=2$, since $\mathbb{S}^1$ can be regarded as a special case of the one-dimensional torus, we refer the reader to the case of tori in \cite{MR4294331}. For $1\leq p\leq \infty$, let $L^p(\mathbb{S}^{d-1})$ be the usual $L^p$ space equipped with the $L^p$ norm. In particular, $L^2(\mathbb{S}^{d-1})$ is a Hilbert space with the inner product $\left\langle f,g\right\rangle:=\int_{\mathbb{S}^{d-1}}fg\text{d}\omega_d$ and the induced norm $\|f\|_{\Lt}: = \sqrt{\left\langle f,f\right\rangle}$. We denote by $C(\mathbb{S}^{d-1})$ the space of continuous functions on $\mathbb{S}^{d-1}$, endowed with the uniform norm $\|f\|_{\infty}:=\text{ess}\sup_{x\in\mathbb{S}^{d-1}}|f(x)|$.

\subsection{Geometric properties of point distributions}

A critical assumption in Assumption \ref{assumption:1MZ} is that the set of $\{(x_j,w_j)\}_{j=1}^m$ is assumed to form an MZ system of order 2 with respect to $\mathbb{P}_{N}$. A natural concern is under what conditions the assumption holds. This assumption is related to the quality of distribution of quadrature points $\mathcal{X}_m:=\{x_j\}_{j=1}^m$. We define the \emph{mesh norm} $h_{\mathcal{X}_m}$ of the quadrature point set $\mathcal{X}_m\subset\mathbb{S}^{d-1}$ as
\begin{equation*}
h_{\mathcal{X}_m} : =\max _{x\in\mathbb{S}^{d-1}}\min_{x_j\in\mathcal{X}_m}\text{dist}(x,x_j),
\end{equation*}
where $\text{dist}(x,y):=\cos^{-1}(x\cdot y)$ is the geodesic distance between $x,y\in\mathbb{S}^{d-1}$. In other words, the mesh norm can be regarded as the geodesic radius of the largest hole in the mesh $\mathcal{X}_m$. Thus, it was investigated in \cite{filbir2011marcinkiewicz,mhaskar2001spherical} that the Assumption (II) in Assumption \ref{assumption:1MZ} holds if
\begin{equation}\label{equ:meshnormMZ}
N\lesssim\frac{\eta}{2h_{\mathcal{X}_{m}}}.
\end{equation}
This assumption holds even when $\mathcal{X}_m$ consists of random points. When the quadrature rule \eqref{equ:quad} is equal-weight, it was shown in \cite{MR2475947} that, if an independent random sample of $m$ points drawn from the distribution $\omega_d$, then there exists a constant $\bar{c}:=\bar{c}(\gamma)$ such that the MZ inequality \eqref{equ:MZproperty} holds with probability exceeding $1-\bar{c}N^{-\gamma}$ on the condition of $m\geq \bar{c} N ^{d-1}\log{N}/\eta^2$.

\subsection{Spherical harmonics and hyperinterpolation}

The restriction to $\mathbb{S}^{d-1}$ of a homogeneous and harmonic polynomial of total degree $\ell$ defined on $\mathbb{R}^{d}$ is called a \emph{spherical harmonic of degree $\ell$} on $\mathbb{S}^{d-1}$. We denote, as usual, by $\{Y_{\ell,k}:k = 1,2,\ldots,Z(d,\ell)\}$ a collection of $\Lt$-orthonormal real-valued spherical harmonics of exact degree $\ell$.
Besides, it is well known (see, e.g., \cite[pp. 38--39]{MR0199449}) that each spherical harmonic $Y_{\ell,k}$ of degree $\ell$ is an eigenfunction of the negative Laplace--Beltrami operator $-\Delta$ for $\mathbb{S}^{d-1}$ with eigenvalue
\begin{equation}\label{equ:LBeigenvale}
\lambda_{\ell}:=\ell(\ell+d-2).
\end{equation}

The family $\{Y_{\ell,k}\}$ of spherical harmonics forms a complete $\Lt$-orthonormal (with respect to $\omega_d$) system for the Hilbert space $\Lt(\mathbb{S}^{d-1})$. Thus, for any $f\in\Lt(\mathbb{S}^{d-1})$, it can be represented by a Laplace--Fourier series
\begin{equation*}
f(x)=\sum_{\ell=0}^{\infty}\sum_{k=1}^{Z(d,\ell)}\hat{f}_{\ell,k} Y_{\ell,k}(x)
\end{equation*}
with coefficients $\hat{f}_{\ell,k}:=\left\langle f,Y_{\ell,k}\right\rangle=\int_{\mathbb{S}^{d-1}}f(x)Y_{\ell,k}(x)\text{d}\omega_d(x)$, $\ell=0,1,2,\ldots$, and $k = 1,2,\ldots,Z(d,\ell)$.

The space $\mathbb{P}_N:=\mathbb{P}_N(\mathbb{S}^{d-1})$ of all spherical polynomials of degree at most $N$ (i.e., the restriction to $\mathbb{S}^{d-1}$ of all polynomials in $\mathbb{R}^{d}$ of degree at most $N$) coincides with the span of all spherical harmonics up to (and including) degree $N$, and its dimension satisfies $\dim\mathbb{P}_N=Z(d+1,N)=\mathcal{O}(N^{d-1})$. The space $\mathbb{P}_N$ is also a reproducing kernel Hilbert space with the reproducing kernel
\begin{equation}\label{equ:kernel}
G_N(x,y) = \sum_{\ell=0}^N\sum_{k=1}^{Z(d,\ell)}Y_{\ell,k}(x)Y_{\ell,k}(y)
\end{equation}
in the sense that $\left\langle \chi,G_N(\cdot,x) \right\rangle = \chi(x)$ for all $\chi\in\mathbb{P}_N(\mathbb{S}^{d-1})$; see, e.g., \cite{MR1115901}. The following lemma, occurred in the proof of Theorem 5.5.2 in \cite{MR1744380}, plays a critical role in our following analysis.
\begin{lemma}[\cite{MR1744380}]\label{lem:2000}
For any given point $x_0\in\mathbb{S}^{d-1}$, there holds
\begin{equation*}
\|G_N(x_0,\cdot)\|_{L^2}^2={Z(d+1,N)}/{\lvert \mathbb{S}^{d-1} \rvert}.
\end{equation*}
\end{lemma}

Given $f\in C(\mathbb{S}^{d-1})$, it is often simpler in practice to express the hyperinterpolant $\mathcal{L}_Nf$ using the reproducing kernel $G_N(\cdot,\cdot)$ defined by \eqref{equ:kernel}. Rearranging the summation, we obtain
\begin{equation*}
\mathcal{L}_Nf(x) = \sum_{\ell=0}^{N}\sum_{k=1}^{Z(d,\ell)}\left(\sum_{j=1}^mw_j f(x_j)Y_{\ell,k}(x_j)\right)Y_{\ell,k}(x) = \sum_{j=1}^mw_jf(x_j)G_N(x,x_j).
\end{equation*}

\begin{lemma}\label{lem:hypernorm}
The norm of the hyperinterpolation operator constructed using quadrature rules \eqref{equ:quad} fulfilling Assumption \ref{assumption:1MZ} in the setting of $C(\mathbb{S}^{d-1})$ to $C(\mathbb{S}^{d-1})$ is bounded by
\begin{equation}\label{equ:uniformnormqmcdesign}
\|\mathcal{L}_N\|_{\infty}:=\sup_{f\in C(\Omega)}\frac{\|\mathcal{L}_Nf\|_{\infty}}{\|f\|_{\infty}}=\mathcal{O}\left(\sqrt{1+\eta}N^{\frac{d-1}{2}}\right).
\end{equation}
\end{lemma}
\begin{proof}
It was derived in \cite{MR1744380} that $\|\mathcal{L}_N\|_{\infty}\leq \lvert \mathbb{S}^{d-1}\rvert ^{1/2} \left(\sum_{j=1}^mw_jG_N(x_0,x_j)^2\right)^{1/2},$ where $x_0\in \mathbb{S}^{d-1}$ is a certain point. Recall that $\{(x_j,w_j)\}_{j=1}^m$ forms an MZ system of order 2 (Assumption II). Thus, we have
\begin{equation*}\begin{split}
\|\mathcal{L}_N\|_{\infty}
\leq& \lvert \mathbb{S}^{d-1}\rvert ^{1/2} \left( (1+\eta)\int_{\mathbb{S}^{d-1}}G_N(x_0,x)^2\text{d}\omega_d(x)    \right)^{1/2}\\
\leq&  \lvert \mathbb{S}^{d-1}\rvert ^{1/2} \sqrt{1+\eta}\|G_N(x_0,\cdot)\|_{L^2} \\
\leq & \sqrt{1+\eta}(\dim\mathbb{P}_N)^{1/2}=\mathcal{O}\left(\sqrt{1+\eta}N^{\frac{d-1}{2}}\right),
\end{split}
\end{equation*}
where in the last inequality we use Lemma \ref{lem:2000}.
\end{proof}

\begin{remark}\label{rem:historical}
The following historical note partly explains the impact of discretizing the inner products \eqref{equ:innerproduct} via some quadrature rules \eqref{equ:quad}. The uniform operator norm of $\mathcal{P}_{N}$ satisfies $\|\mathcal{P}_N\|_{\infty} \asymp N^{\frac{d-2}{2}}$,
where $a_N\asymp b_N$ denotes that there exist $c_1,c_2>0$ independent of $N$ such that $c_1a_N\leq b_N\leq c_2b_N$, and the case of $\mathbb{S}^2$ $(d=3)$ can be dated back to Gronwall \cite{MR1500962}. However, the uniform norm $\|\mathcal{L}_N\|_{\infty}$ of the hyperinterpolation operator constructed using quadrature rules \eqref{equ:quad} with quadrature exactness \eqref{equ:exactness} is bounded as $\|\mathcal{L}_N\|_{\infty}=\mathcal{O}(n^{\frac{d-1}{2}})$. That is, the growth rate of the uniform norm $\|\mathcal{L}_N\|_{\infty}$ of the hyperinterpolation operator with quadrature exactness \eqref{equ:exactness}, as shown in \cite{MR1744380}, is worse by a factor of $n^{1/2}$ than the optimal result for $\mathcal{P}_N$. Only for the special case of $d=3$ and under a mild additional assumption on the quadrature rule \eqref{equ:quad}, the improved result of $\|\mathcal{L}_{N}\|_{\infty} \asymp n^{1/2}$ was achieved in \cite{MR1744380}.
\end{remark}

\subsection{Sobolev spaces}
The study of hyperinterpolation in a Sobolev space setting can be traced back to the work \cite{MR2274179} by Hesse and Sloan. We define the Sobolev space for $s\geq 0$ as the set of all functions $f\in L^2(\mathbb{S}^{d-1})$ whose Laplace--Fourier coefficients satisfy
\begin{equation*}
\sum_{\ell=0}^{\infty}\sum_{k=1}^{Z(d,\ell)}(1+\lambda_{\ell})^s\lvert \hat{f}_{\ell,k}\rvert^2<\infty,
\end{equation*}
where $\lambda_{\ell}$ is given as \eqref{equ:LBeigenvale}. When $s=0$, we have $H^0(\mathbb{S}^{d-1})=L^2(\mathbb{S}^{d-1})$. The norm in $\Hs(\mathbb{S}^{d-1})$ is therefore defined as
\begin{equation*}
\|f\|_{\Hs}:=\left(\sum_{\ell=0}^{\infty}\sum_{k=1}^{Z(d,\ell)}(1+\lambda_{\ell})^s\lvert \hat{f}_{\ell,k}\rvert^2\right)^{1/2}.
\end{equation*}

The following lemma is necessary for our analysis, which was first presented in \cite{MR2274179}.
\begin{lemma}\label{lem:hsl2}
For any $f\in\mathbb{P}_N$, $\|f\|_{\Hs}\leq c_1 N^s\|f\|_{\Lt}$, where $c_1>0$ is a constant.
\end{lemma}

Denote $\mathcal{L}_{>N}:=I-\mathcal{L}_N$. Based on Lemma \ref{lem:hsl2}, we study $\mathcal{L}_{>N}$ in the $\|\cdot\|_{\infty}$ sense in the following lemma.

\begin{lemma}\label{lem:Pi>N}
 Given $f\in C(\mathbb{S}^{d-1})$ and $t>\frac{d-1}{2}$, the stability of $\mathcal{L}_{>N}$ as an operator from $C(\mathbb{S}^{d-1})$ to $C(\mathbb{S}^{d-1})$ can be controlled by
\begin{equation*}
\|\mathcal{L}_{>N}f\|_{\infty}\leq\left(1+\|\mathcal{L}_N\|_{\infty}\right)E_N(f)+c_2\eta N^{t}\|\chi^*\|_{\Lt},
\end{equation*}
where $c_2>0$ is a constant only depending on $\eta$, $E_{N}(f)=\inf_{\chi\in\mathbb{P}_{N}}\|f-\chi\|_{\infty}$ denotes the best uniform approximation error of $f$ in $\mathbb{P}_{N}$, $\chi^*\in\mathbb{P}_N$ is the best approximation of $f$ in $\mathbb{P}_N$ such that $\|f-\chi^*\|_{\infty}=E_N(f)$. Furthermore, we have
\begin{equation}\label{equ:hyperinftyerror2}
\|\mathcal{L}_{>N}f\|_{\infty}\leq\left(1+\|\mathcal{L}_N\|_{\infty}+c_2\eta N^{t}\right)E_N(f)+c_2\eta N^{t}\|f\|_{\infty}.
\end{equation}
\end{lemma}
\begin{proof}
For any $\chi\in\mathbb{P}_{N}$, we have $\mathcal{L}_{>N}f=f-\mathcal{L}_{N}f = f-\chi-\mathcal{L}_{N}(f-\chi) - (\mathcal{L}_N\chi -\chi)$, and hence
\begin{equation*}\begin{split}
\|f-\mathcal{L}_{N}f\|_{\infty}\leq\|f-\chi\|_{\infty}+\|\mathcal{L}_{N}\|_{\infty}\|f-\chi\|_{\infty} + \|\mathcal{L}_N\chi -\chi\|_{\infty}.
\end{split}\end{equation*}
Since this holds for arbitrary $\chi\in\mathbb{P}_{N}$, we have
\begin{equation*}
\|\mathcal{L}_N\chi -f\|_{\infty}\leq\left(1+\|\mathcal{L}_N\|_{\infty}\right)E_N(f) + \|\mathcal{L}_N\chi^* -\chi^*\|_{\infty}.
\end{equation*}
Then, we control the term $\|\mathcal{L}_N\chi^* -\chi^*\|_{\infty}$ with the aid of the Sobolev embedding of $H^t(\mathbb{S}^{d-1})$ into $C(\mathbb{S}^{d-1})$ for any $t>\frac{d-1}{2}$. Note that $\|\mathcal{L}_N\chi-\chi\|_{L^2}^2\leq(\eta^2+4\eta)\|\chi\|_{L^2}^2$ for any $\chi\in\mathbb{P}_N$, which was proved in \cite{an2022bypassing}. Thus, it follows from Lemma \ref{lem:hsl2} that
\begin{equation*}\begin{split}
\|\mathcal{L}_N\chi^* -\chi^*\|_{\infty}&\lesssim \|\mathcal{L}_N\chi^* -\chi^*\|_{H^t} \leq c_1N^t \|\mathcal{L}_N\chi^* -\chi^*\|_{L^2}\\
&\leq  c_1N^t\sqrt{\eta^2+4\eta} \|\chi^*\|_{L^2}\leq c_2\eta N^t.
\end{split}\end{equation*}
The estimate \eqref{equ:hyperinftyerror2} is immediately obtained by noting that $\|\chi^*\|_{\infty}\leq \|f\|_{\infty} + E_N(f)$.
\end{proof}

\section{$L^{\infty}$ stability and effective maximum principle}\label{sec:theory}

We now study the $L^{\infty}$ stability and effective maximum principle of the spectral scheme \eqref{equ:scheme} with quadrature rules \eqref{equ:quad} fulfilling Assumption \ref{assumption:1MZ} for the Allen--Cahn equation \eqref{equ:AC} on $\mathbb{S}^{d-1}\subset\mathbb{R}^d$. A key observation is that, for $f\in H^{s}(\mathbb{S}^{d-1})$ with $s> \frac{d-1}{2}$, the best approximation error $E_N(f)$ in $\mathbb{P}_N$ can be bounded as
\begin{equation*}
E_N(f)\leq \frac{c_3(f)}{N^{s-\frac{d-1}{2}}}\|f\|_{H^s},
\end{equation*}
where $c_3(f)>0$ is some constant depending on $f$. Such an error rate can be obtained by \cite{MR288468}, together with the Sobolev embedding into H\"older spaces.

\subsection{The case of $0<\tau\leq 1/2$}
We first consider the case of $0<\tau\leq 1/2$.

\begin{theorem}[$L^{\infty}$ stability for $0<\tau\leq 1/2$]\label{thm:uniformstability1}
Let $0<\alpha_0\leq 1$, $0< \tau\leq 1/2$, and $s_0$ be a constant marginally larger than $(d-1)/2$. Assume $u_0\in H^{s}(\mathbb{S}^{d-1})$ with $s> d-1$ and $\|u_0\|_{\infty}\leq 1$. If $\eta= \tilde{c} N^{-\varepsilon}$ for any $\tilde{c}\geq0$ and $\varepsilon >s_0$ and $N\geq N_1:=N_1\left(\alpha_0,\nu,s,d,u_0,\varepsilon\right)$, then
\begin{equation*}
\sup_{n\geq 0} \|u^n\|_{\infty}\leq 1+\alpha_0.
\end{equation*}
\end{theorem}
\begin{proof}
This theorem is proved by induction.

Step 1: Initial data. As $\mathcal{L}_{>N}=I-\mathcal{L}_N$, by Lemmas \ref{lem:hypernorm} and \ref{lem:Pi>N} we have
\begin{equation*}\begin{split}
\|\mathcal{L}_Nu_0\|_{\infty}
\leq & \|u_0\|_{\infty}+\|\mathcal{L}_{>N}u_0\|_{\infty}\leq  1+(1+\|\mathcal{L}_N\|_{\infty})E_N(u_0)+c_2\eta N^{s_0}\|\chi^*\|_{\Lt}\\
\leq & 1+c_3\left(1+\sqrt{1+\eta}(\dim{\mathbb{P}_N})^{1/2})\right)N^{-(s-\frac{d-1}{2})}\|u_0\|_{H^{s}} +c_2\eta N^{s_0}\|\chi^*\|_{\Lt}\\
\leq & 1+\tilde{c}_3(1+\sqrt{2}N^{\frac{d-1}{2}})N^{-(s-\frac{d-1}{2})}\|u_0\|_{H^{s}} +c_2\tilde{c} N^{-\varepsilon+s_0}\|\chi^*\|_{\Lt}\\
\leq & 1+\alpha_0,
\end{split}\end{equation*}
if $N\geq N_1'(\alpha_0,s,d,\|u_0\|_{H^{s}},\varepsilon)$ is large enough such that
\begin{equation}\label{equ:positivealpha}
\tilde{c}_3(1+\sqrt{2}N^{\frac{d-1}{2}})N^{-(s-\frac{d-1}{2})}\|u_0\|_{H^{s}} +c_2\tilde{c} N^{-\varepsilon+s_0}\|\chi^*\|_{\Lt}\leq \alpha_0,
\end{equation}
where $\tilde{c}_3>0$ is a constant stemming from $\dim{\mathbb{P}}_N=\mathcal{O}(N^{d-1})$.

Step 2: Induction. The inductive assumption is $\|u^n\|_{\infty}\leq 1+\alpha_0$. We intend to show 
\begin{equation*}
u^{n+1}\leq 1+\alpha_0.
\end{equation*}
Afterwards, repeating the argument for $-u^{n+1}$ gives $-(1+\alpha_0)\leq u^{n+1}$. Thus, we have $\|u^{n+1}\|_{\infty}\leq 1+\alpha_0$. Note that the scheme \eqref{equ:scheme} is equivalent to
\begin{equation*}
(1-\tau\nu^2\Delta)u^{n+1}=u^n + \tau\mathcal{L}_N\left(u^n-(u^n)^3\right).
\end{equation*}
Denote $u^{n}:=1+\zeta^{n}$. Thus, the inductive assumption implies 
\begin{equation*}
-(2+\alpha_0)\leq \zeta^n\leq \alpha_0.
\end{equation*}
For $\zeta^{n+1}:=u^{n+1}-1$, we have
\begin{equation*}\begin{split}
(1-\tau\nu^2\Delta)\zeta^{n+1}
& = \zeta^n +\tau\mathcal{L}_N\left((1+\zeta^n)-(1+\zeta^n)^3\right)  \\
& = \zeta^n + \tau\mathcal{L}_N\left( -2\zeta^n-3(\zeta^n)^2-(\zeta^n)^3\right)\\
& = \zeta^n + \tau\left( -2\zeta^n-3(\zeta^n)^2-(\zeta^n)^3 -    \mathcal{L}_{>N}( -2\zeta^n-3(\zeta^n)^2-(\zeta^n)^3)\right)\\
&= (1-2\tau)\zeta^n  - \tau(\zeta^n)^2(3+\zeta^n) + \tau\mathcal{L}_{>N}(2\zeta^n+3(\zeta^n)^2+(\zeta^n)^3).
\end{split}\end{equation*}
To proceed, we first note that $3+\zeta^n\geq 1-\alpha_0\geq 0$. Besides, since $\sup_{1\leq j\leq n}\|u^{j}\|_{\infty}\leq 1+\alpha_0$, we have $\sup_{1\leq j\leq n}(\|u^{j}\|_{L^2}+\|\mathcal{L}_N(f(u^j))\|_{L^2})$ is bounded by a fixed constant, and then $\sup_{1\leq j\leq n}\|u^{j}\|_{H^s}\leq c_{\nu,u_0,s,d}$,
where $c_{\nu,u_0,s,d}$ is some constant depending only on $\mu$, $u_0$, $s$ and $d$. This bound on $\|u^{j}\|_{H^s}$ can be shown by using the discrete smoothing estimate (cf. \cite{li2021stability}) to the following iterated scheme
\begin{equation*}\begin{split}
u^{n+1}
&= (I-\tau\nu^2\Delta)^{-1}u^n -(I-\tau\nu^2\Delta)^{-1}\tau\mathcal{L}_N(f(u^n))\\
&=: T_0u^n -\tau T_0\mathcal{L}_N(f(u^n))\\
&= T_0^{J+1}u^{n-J} -\tau\sum_{j=1}^JT_0^{j+1}\mathcal{L}_N(f(u^{n-j})) - \tau T_0\mathcal{L}_N(f(u^n)),
\end{split}\end{equation*}
where $T_0 := (I-\tau\nu^2\Delta)^{-1}$. As the following analysis ensures $\sup_{1\leq j\leq n+1}\|u^{j}\|_{\infty}\leq 1+\alpha_0$, we also have $\|u^{n+1}\|_{H^s}\leq c_{\nu,u_0,s,d}$ in the next iteration. By the maximum principle and Lemmas \ref{lem:hypernorm} and \ref{lem:Pi>N}, we have

\begin{equation*}\begin{split}
&\max \zeta^{n+1}\\
& \leq (1-2\tau)\alpha_0 + \tau \left\| \mathcal{L}_{>N}(2\zeta^n+3(\zeta^n)^2+(\zeta^n)^3)\right\|_{\infty}\\
& \leq (1-2\tau)\alpha_0 + \tau\left[\left(1+ \sqrt{1+\eta}N^{\frac{d-1}{2}} + c_2\eta N^{s_0}\right)N^{-s+\frac{d-1}{2}}\|2\zeta^n+3(\zeta^n)^2+(\zeta^n)^3\|_{H^s}\right.\\
&\quad\left.+c_2\eta N^{s_0}\|2\zeta^n+3(\zeta^n)^2+(\zeta^n)^3\|_{\infty}\right]\\
&\leq (1-2\tau)\alpha_0 + \tau\left(N^{-(s-\frac{d-1}{2})}+ N^{d-1-s} +  N^{\frac{d-1}{2}+s_0 -s-\varepsilon} +  N^{s_0-\varepsilon}\right) \text{const}(\nu,u_0,s,d)\\
&\leq \alpha_0,
\end{split}\end{equation*}
if $N\geq N_1''(\alpha_0,\nu,s,d,\varepsilon)$ is large enough, which then leads to $\max u^{n+1}\leq 1+\alpha_0$. Thus, this theorem is proved by letting $N_1=\max\{N_1',N_1''\}$.
\end{proof}

\begin{remark}
The situation \eqref{equ:positivealpha} in our proof requires $\alpha_0>0$. This requirement stems from the fact that the hyperinterpolation operator does not preserve the sharp uniform bound. That is, $\|u_0\|_{\infty}\leq 1$ does not \emph{necessarily} imply $\|\mathcal{L}_Nu_0\|_{\infty}\leq 1$; see Remark \ref{rem:historical}. This fact also explains why spectral methods involving orthogonal projection and hyperinterpolation do not necessarily preserve the sharp maximum principle. In the following theorems, we may directly assume that $\|\mathcal{L}_Nu_0\|_{\infty}\leq 1+\alpha_0$ for $0\leq \alpha_0\leq1$. Thus, $\alpha_0=0$ is possible because the spectral error term brought by hyperinterpolation has been explicitly recorded in such an assumption.
\end{remark}

\begin{theorem}[Effective maximum principle for $0<\tau\leq 1/2$]\label{thm:effectivemaxprinciple1}
Let $0<\tau\leq 1/2$ and $s_0$ be a constant marginally larger than $(d-1)/2$. Assume $u_0\in H^{s}(\Omega)$ with $s>d-1$ and $\|u^0\|_{\infty}\leq 1+\alpha_0$ for some $0\leq\alpha_0\leq 1$. If $\eta= \tilde{c} N^{-\varepsilon}$ for any $\tilde{c}\geq0$ and $\varepsilon >s_0$ and $N\geq N_2 :=N_2(\nu,s,d,u_0,\varepsilon)$, then, for any $n\geq 1$,
\begin{equation}\label{equ:max1}
\|u^n\|_{\infty}\leq 1+\theta^n\alpha_0 + \frac{1-\theta^n}{1-\theta}\tau C_{\nu,u_0,s,d}\left(\sqrt{1+\eta}N^{d-1-s} + \eta N^{s_0+\frac{d-1}{2}-s} + \eta N^{s_0}\right),
\end{equation}
where $\theta = 1-2\tau$, and $C_{\nu,u_0,s,d}>0$ is a constant depending on $\nu$, $u_0$, $s$, and $d$. Consequently,
\begin{equation*}
\limsup_{n\rightarrow\infty} \|u^n\|_{\infty}\leq 1+\frac{1}{2}C_{\nu,u_0,s,d}\left(\sqrt{1+\eta}N^{d-1-s} + \eta N^{s_0+\frac{d-1}{2}-s} + \eta N^{s_0}\right)
\end{equation*}
and
\begin{equation*}
\limsup_{N\rightarrow\infty} \|u^n\|_{\infty}\leq 1+\theta^n\alpha_0.
\end{equation*}
\end{theorem}
\begin{proof}
By the inductive step in the proof of Theorem \ref{thm:uniformstability1}, there exists $N_2:=N_2(\nu,s,d,u_0,\varepsilon)>0$ such that for any $N\geq N_2$, we have the weakest estimate $\sup_{n\geq 0}\|u^n\|_{\infty}\leq2$. Denote $u^n:=1+\zeta^n$ and define $\alpha_n:=\max\zeta^n$. Then, by repeating the procedure in the proof of Theorem \ref{thm:uniformstability1}, we have
\begin{equation*}
\alpha_{n+1} \leq   (1-2\tau)\alpha_n +\tau C_{\nu,u_0,s,d} \left(\sqrt{1+\eta}N^{d-1-s} + \eta N^{s_0+\frac{d-1}{2}-s} + \eta N^{s_0}\right),
\end{equation*}
where the constant $C_{\nu,u_0,s,d}>0$ depends on $\nu$, $u_0$, $s$, and $d$. A similar estimate also holds for $\tilde{\alpha}_n:=\max(-1-u^n)$. Thus, for $\theta = 1-2\tau$, iterating in $n$ then gives the effective maximum principle \eqref{equ:max1}. Letting $n\rightarrow\infty$ and $N\rightarrow\infty$ leads to both limit cases, respectively.
\end{proof}

\subsection{The case of $1/2<\tau<2$}
We now consider the case of $1/2< \tau<2$, with the aid of a prototype iterative system investigated in \cite[Lemma 3.3]{MR4294331}.
\begin{lemma}[Prototype iterative system for the maximum principle \cite{MR4294331}]\label{lem:dongli}
Let $0<\tau<2$ and $p(x)=(1+\tau)x-\tau x^3$. Consider the recurrent relation
\begin{equation*}
\alpha_{n+1}:=\max_{|x|\leq \alpha_n}|p(x)|+\zeta,\quad n \geq 0,
\end{equation*}
where $\zeta>0$.
\begin{enumerate}
\item Case $0<\tau\leq1/2$. Let $\alpha_0=2$. There exists an absolute constant $\zeta_0>0$ sufficiently small such that for all $0< \zeta\leq \zeta_0$, we have $1\leq \alpha_n\leq 2$ for all $n$.
\item Case $1/2<\tau\leq 2-\epsilon_0$ for some $0<\epsilon_0\leq 1$.
Let
\begin{equation*}
\alpha_0 = \frac12\left(\frac{(1+\tau)^{3/2}}{\sqrt{3\tau}}\cdot\frac23+\sqrt{\frac{2+\tau}{\tau}}\right).
\end{equation*}
Then, there exists a constant $\zeta_0>0$ depending only on $\epsilon_0$ such that if $0< \zeta\leq \zeta_0$, then, for all $n\geq 1$, we have
\begin{equation*}
\frac{(1+\tau)^{3/2}}{\sqrt{3\tau}}\cdot\frac23+\zeta\leq\alpha_n\leq \alpha_0.
\end{equation*}
\end{enumerate}
\end{lemma}

\begin{remark}
For $\tau\geq 2$, such a stability result does not hold; see counterexamples provided in Remark 3.7 and Corollary 3.1 in \cite{MR4294331}.
\end{remark}

\begin{theorem}[$L^{\infty}$ stability for $1/2<\tau<2$]\label{thm:uniformstability2}
Let $1/2<\tau\leq 2-\epsilon_0$ for some $0<\epsilon_0\leq 1$,
\begin{equation*}
M_0 = \frac12\left(\frac{(1+\tau)^{3/2}}{\sqrt{3\tau}}\cdot\frac23+\sqrt{\frac{2+\tau}{\tau}}\right),
\end{equation*}
and $s_0$ be a constant marginally larger than $(d-1)/2$. Assume $u_0\in H^{s}(\mathbb{S}^{d-1})$ with $s>d-1$ and $\|u^0\|_{\infty}\leq M_0$. If $\eta= \tilde{c} N^{-\varepsilon}$ for any $\tilde{c}\geq0$ and $\varepsilon >s_0$ and $N\geq N_3:=N_3\left(\epsilon_0,\nu,s,d,u_0,\varepsilon\right)$, then we have
\begin{equation*}
\sup_{n\geq 0}\|u^n\|_{\infty}\leq M_0.
\end{equation*}
\end{theorem}
\begin{remark}\label{rem:dongli}
As suggested in \cite{MR4294331}, the bound $M_0$ can be replaced with any number
\begin{equation}\label{equ:m0}
\tilde{M}_0\in\left(\frac{(1+\tau)^{3/2}}{\sqrt{3\tau}},\sqrt{\frac{2+\tau}{\tau}}\right).
\end{equation}
Correspondingly, $N_3$ in Theorem \ref{thm:uniformstability2} should also depend on $\tilde{M}_0$, or more precisely, on its distance to the end points of the interval in \eqref{equ:m0}.
\end{remark}

\begin{proof}
We adopt the same induction setting in the proof of Theorem \ref{thm:uniformstability1}. Note that
\begin{equation*}\begin{split}
&(1-\tau\nu^2\Delta)u^{n+1}\\
 =& u^n - \mathcal{L}_Nu^n + \mathcal{L}_N\left((1+\tau)u^n-\tau(u^n)^3\right) \\
=& (u^n - \mathcal{L}_Nu^n) +\left((1+\tau)u^n-\tau(u^n)^3\right) - \mathcal{L}_{>N}\left((1+\tau)u^n-\tau(u^n)^3\right).
\end{split}\end{equation*}
Recall $p(x)=(1+\tau)x-\tau x^3$. Then, by the maximum principle, we have
\begin{equation*}\begin{split}
\|u^{n+1}\|_{\infty}\leq&\|u^n - \mathcal{L}_Nu^n\|_{\infty}+\|p(u^n)\|_{\infty}+\|\mathcal{L}_{>N}\left((1+\tau)u^n-\tau(u^n)^3\right)\|_{\infty}\\
\leq&  \|u^n - \mathcal{L}_Nu^n\|_{\infty} + M_0-\zeta+ \|\mathcal{L}_{>N}\left((1+\tau)u^n-\tau(u^n)^3\right)\|_{\infty},
\end{split}\end{equation*}
where the estimates for $\|u^n - \mathcal{L}_Nu^n\|_{\infty}$ and $\|\mathcal{L}_{>N}\left((1+\tau)u^n-\tau(u^n)^3\right)\|_{\infty}$ are similar to that in the proof of Theorem \ref{thm:effectivemaxprinciple1}. Thus, the theorem follows from Lemma \ref{lem:dongli} and induction.
\end{proof}

\section{Refined results with quadrature exactness}\label{sec:discussion}

In this section, we demonstrate that our scheme \eqref{equ:scheme} is equivalent to the discrete Galerkin scheme if the quadrature exactness \eqref{equ:exactness} is assumed. With such an assumption, we can also investigate the energy stability of our scheme \eqref{equ:scheme}, which is not mentioned in Section \ref{sec:theory}.

\subsection{Discrete Galerkin method}

It should be noted that, although this paper only focuses on the Allen--Cahn equation \eqref{equ:AC}, such equivalence also holds for other semi-linear partial differential equations \eqref{equ:PDE}, namely,
\begin{equation*}
u_t=\mathbf{L}u+\mathbf{N}(u).
\end{equation*}
In spirit of our scheme \eqref{equ:scheme}, we consider the following semi-discrete scheme for the semi-linear PDE \eqref{equ:PDE}:
\begin{equation}\label{equ:semischeme}
\dfrac{u^{n+1}-u^n}{\tau}=\mathbf{L}u^{n+1}+\mathcal{L}_N(\mathbf{N}(u^n)),
\end{equation}
where $\tau>0$ is the time stepping size, and $u^n\in\mathbb{P}_N$ denotes the numerical solution at $t=n\tau$. If the quadrature exactness \eqref{equ:exactness} is assumed, then the hyperinterpolation operator $\mathcal{L}_N$ is a discrete projection operator in the sense of
\begin{equation}\label{equ:discreteprojection}
\langle f-\mathcal{L}_Nf ,\chi\rangle_m =0\quad\text{and}\quad\mathcal{L}_N\chi = \chi \quad\forall\chi\in\mathbb{P}_N,
\end{equation}
as originally shown in \cite{sloan1995polynomial}. The scheme \eqref{equ:semischeme} is equivalent to
\begin{equation*}
\mathcal{L}_N\left(\dfrac{u^{n+1}-u^n}{\tau}-\mathbf{L}u^{n+1}-\mathbf{N}(u^n)\right) = \dfrac{u^{n+1}-u^n}{\tau}-\mathbf{L}u^{n+1}-\mathcal{L}_N(\mathbf{N}(u^n))=0,
\end{equation*}
which can be obtained using the linearity of $\mathcal{L}_N$ and the property \eqref{equ:discreteprojection}. Then, with the property \eqref{equ:discreteprojection} again, we know
\begin{equation*}
\left\langle \dfrac{u^{n+1}-u^n}{\tau}-\mathbf{L}u^{n+1}-\mathbf{N}(u^n),\chi \right\rangle_m = 0\quad \forall \chi\in\mathbb{P}_N,
\end{equation*}
which is further equivalent to
\begin{equation}\label{equ:qualocation}
\frac{1}{\tau}\left\langle u^{n+1}-u^n,\chi \right\rangle_m = \left\langle \mathbf{L}u^{n+1},\chi \right\rangle_m + \left\langle \mathbf{N}(u^n), \chi  \right\rangle_m \quad \forall \chi \in \mathbb{P}_N.
\end{equation}
Note that (\ref{equ:qualocation}) is the \emph{discrete} Galerkin method for the scheme \eqref{equ:semischeme} on the quadrature points $\mathcal{X}_m$.

Focusing on the Allen--Cahn equation \eqref{equ:AC}, the above discussion suggests that, if the quadrature exactness \eqref{equ:exactness} is assumed, then the proposed scheme \eqref{equ:scheme} is equivalent to
\begin{equation}\label{equ:ACqualocation}
\frac{1}{\tau}\left\langle u^{n+1}-u^n,\chi \right\rangle_m = \left\langle\nu^2\Delta u^{n+1},\chi \right\rangle_m - \left\langle (u^n)^3-u^n, \chi  \right\rangle_m \quad \forall \chi \in \mathbb{P}_N,
\end{equation}
with $u^0 = \mathcal{L}_Nu_0\in\mathbb{P}_N$. The schemes \eqref{equ:qualocation} and \eqref{equ:ACqualocation} describe a quadrature-based Galerkin method, and it may be also known as the \emph{qualocation method}, or more precisely, \emph{quadrature-modified collocation method}, as firstly investigated by Sloan and Wendland in \cite{MR960849,MR1027125}. The motivation of the qualocation method is to design numerical schemes achieving the theoretical benefits of the Galerkin method at a computational cost comparable to the collocation method.

\subsection{Refined results}

An immediate consequence of the quadrature exactness \eqref{equ:exactness} is $\eta = 0$. Thus, we have the following corollary of the theorems in Section \ref{sec:theory}. Note that if $\eta = 0$, then $N_1$, $N_2$, and $N_3$ do not necessarily depend on $\varepsilon$.
\begin{corollary}\label{cor:quadrature}
Consider the scheme \eqref{equ:scheme} for the Allen--Cahn equation \eqref{equ:AC} on $\mathbb{S}^{d-1}$, where the quadrature rule \eqref{equ:quad} has exactness degree at least $2N$. Assume $u_0\in H^{s}(\mathbb{S}^{d-1})$ with $s>d-1$. Then, the following assertions holds:
\begin{enumerate}
    \item  \emph{$L^{\infty}$ stability for $0<\tau\leq 1/2$.} Let $0<\alpha_0\leq 1$ and $0< \tau\leq 1/2$. Assume $\|u_0\|_{\infty}\leq 1$. If $N\geq N_4:=N_4\left(\alpha_0,\nu,s,d,u_0\right)$, then
\begin{equation*}
\sup_{n\geq 0} \|u^n\|_{\infty}\leq 1+\alpha_0.
\end{equation*}
    \item \emph{Effective maximum principle for $0<\tau \leq 1/2$.} Let $0<\tau\leq 1/2$. Assume $\|u^0\|_{\infty}\leq 1+\alpha_0$ for some $0<\alpha_0\leq 1$. If $N\geq N_4':=N_4'(\nu,s,d,u_0)$, then, for any $n\geq 1$,
\begin{equation*}
\|u^n\|_{\infty}\leq 1+\theta^n\alpha_0 + \frac{1-\theta^n}{1-\theta}\tau C_{\nu,u_0,s,d}N^{d-1-s},
\end{equation*}
where $\theta = 1-2\tau$, and $C_{\nu,u_0,s,d}>0$ is a constant depending on $\nu$, $u_0$, $s$, and $d$.
    \item \emph{$L^{\infty}$-stability for $1/2< \tau< 2$.} Let $1/2<\tau< 2-\epsilon_0$ for some $0<\epsilon_0\leq 1$, and let
\begin{equation*}
M_0 = \frac12\left(\frac{(1+\tau)^{3/2}}{\sqrt{3\tau}}\cdot\frac23+\sqrt{\frac{2+\tau}{\tau}}\right).
\end{equation*}
 Assume $\|u^0\|_{\infty}\leq M_0$.
 If $N\geq N_4'':=N_4''(\epsilon_0,\nu,s,d,u_0)$, then
\begin{equation*}
\sup_{n\geq 0}\|u^n\|_{\infty}\leq M_0.
\end{equation*}
\end{enumerate}
\end{corollary}

\begin{remark}
Recall the historical note in Remark \ref{rem:historical}. If we consider the Allen--Cahn equation \eqref{equ:AC} on $\mathbb{S}^2$, then the order of $\|\mathcal{L}_N\|_{\infty}$ can be reduced by $n^{1/2}$, and the results in Corollary \ref{cor:quadrature} can be improved correspondingly.
\end{remark}

We then consider the energy stability of the numerical solutions in presence of quadrature exactness. Recall that the energy functional $\mathcal{E}(u)$ of $u$ is defined as \eqref{equ:energy}, and its discrete version can be defined as
\begin{equation}\label{equ:discreteenergy}
\tilde{\mathcal{E}}(u) :=\sum_{j=1}^mw_j \left(\frac{\nu^2}{2}(\nabla u(x_j)\cdot\nabla u(x_j)) + F(u(x_j))\right),
\end{equation}
which is discretized by the quadrature rule \eqref{equ:quad}. Besides, recall that all weights $w_j$ are positive.

\begin{lemma}[Energy estimate]\label{lem:energyesti}
For any $n\geq 0$, if the quadrature rule \eqref{equ:quad} has exactness degree $2N$, then, the sequence $\{u^n\}_{n\geq0}$ generated by the scheme \eqref{equ:scheme} satisfies
\begin{equation}\label{equ:discreteenergyesti}\begin{split}
&\tilde{\mathcal{E}}(u^{n+1})-\tilde{\mathcal{E}}(u^n)+\left(\frac{1}{\tau}+\frac12\right)\sum_{j=1}^mw_j(u^{n+1}(x_j)-u^n(x_j))^2 \\
&\quad\quad\quad\quad \leq \frac32 \max\left\{\|u^n\|_{\infty}^2,\|u^{n+1}\|_{\infty}^2\right\}\sum_{j=1}^mw_j(u^{n+1}(x_j)-u^n(x_j))^2,
\end{split}\end{equation}
where the discrete energy $\tilde{\mathcal{E}}(u)$ of $u$ is given by \eqref{equ:discreteenergy}. Furthermore, if the quadrature rule \eqref{equ:quad} has exactness degree $4N$, then the sequence $\{u^n\}_{n\geq0}$ generated by the scheme \eqref{equ:scheme} satisfies
\begin{equation}\label{equ:energyesti}
\begin{split}
&\mathcal{E}(u^{n+1})-\mathcal{E}(u^n)+\left(\frac{1}{\tau}+\frac12\right)\int_{\mathbb{S}^{d-1}}(u^{n+1}-u^n)^2{\rm{d}}\omega_d \\
 & \quad\quad\quad\quad \leq \frac32 \max\left\{\|u^n\|_{\infty}^2,\|u^{n+1}\|_{\infty}^2\right\}\int_{\mathbb{S}^{d-1}}(u^{n+1}-u^n)^2{\rm{d}}\omega_d,
\end{split}\end{equation}
where the energy $\mathcal{E}(u)$ of $u$ is given by \eqref{equ:energy}.
\end{lemma}

\begin{proof}
Note that
\begin{equation}\label{equ:energylem1}\begin{split}
\frac{1}{\tau}\int_{\mathbb{S}^{d-1}}(u^{n+1}-u^n)^2{\rm{d}}\omega_d
& = \left\langle\frac{u^{n+1}-u^n}{\tau},u^{n+1}-u^n\right\rangle\\
&= \left\langle\nu^2\Delta u^{n+1}-\mathcal{L}_{N}\left((u^n)^3-u^n\right), u^{n+1}-u^n\right\rangle\\
&=\nu^2\left\langle \Delta u^{n+1},u^{n+1}-u^n\right\rangle-\left\langle \mathcal{L}_{N}\left(f(u^n)\right),u^{n+1}-u^n\right\rangle.
\end{split}\end{equation}
For the first term on the right-hand side of \eqref{equ:energylem1}, the Green--Beltrami identity suggests
\begin{equation}\label{equ:energylem2}\begin{split}
&\nu^2\left\langle \Delta u^{n+1},u^{n+1}-u^n\right\rangle= -\nu^2\int_{\mathbb{S}^{d-1}}\nabla u^{n+1}\cdot\nabla(u^{n+1}-u^n)\text{d}\omega_d\\
=&-\frac{\nu^2}{2}\left(\int_{\mathbb{S}^{d-1}} \lvert\nabla u^{n+1}\rvert^2\text{d}\omega_d - \int_{\mathbb{S}^{d-1}} \lvert \nabla u^{n}\rvert^2\text{d}\omega_d +
\int_{\mathbb{S}^{d-1}} \lvert \nabla (u^{n+1}-u^n)\rvert^2\text{d}\omega_d\right).
\end{split}\end{equation}
Note that all the integrands in the integrals and inner products (regarded as integrals) in the above expressions \eqref{equ:energylem1} and \eqref{equ:energylem2} are polynomials of degree at most $2N$. These integrals and inner products can be replaced by their discrete versions \eqref{equ:quad} and \eqref{equ:discreteinnerproduct}, respectively, with the assumption that the quadrature exactness degree be $2N$ or $4N$.

Meanwhile, as 
\begin{equation*}
F(u^{n+1})=F(u^n) + f(u^n)(u^{n+1}-u^n)+\frac12f'(\xi)(u^{n+1}-u^n)^2,
\end{equation*}
where $\xi$ lies between $u^n$ and $u^{n+1}$, we then have
\begin{equation}\label{equ:energylem5}\begin{split}
\sum_{j=1}^mw_j F(u^{n+1}(x_j))&\leq\sum_{j=1}^mF(u^n(x_j)) + \left\langle f(u^n),u^{n+1}-u^n\right\rangle_m\\
&+\left(\frac32 \max\left\{\|u^n\|_{\infty}^2,\|u^{n+1}\|_{\infty}^2\right\}-\frac12\right)\sum_{j=1}^mw_j(u^{n+1}(x_j)-u^n(x_j))^2
\end{split}\end{equation}
and
\begin{equation}\label{equ:energylem4}\begin{split}
\int_{\mathbb{S}^{d-1}} F(u^{n+1}) \text{d}\omega_d&\leq \int_{\mathbb{S}^{d-1}}F(u^n) \text{d}\omega_d + \left\langle f(u^n),u^{n+1}-u^n\right\rangle\\
&+\left(\frac32 \max\left\{\|u^n\|_{\infty}^2,\|u^{n+1}\|_{\infty}^2\right\}-\frac12\right)\int_{\mathbb{S}^{d-1}}(u^{n+1}-u^n)^2\text{d}\omega_d.
\end{split}\end{equation}

When the quadrature exactness degree is $2N$, we know from \eqref{equ:energylem5} and the discrete versions of equations \eqref{equ:energylem1} and \eqref{equ:energylem2} that
\begin{equation*}\begin{split}
&\tilde{\mathcal{E}}(u^{n+1})-\tilde{\mathcal{E}}(u^n)+\frac{\nu^2}{2}\sum_{j=1}^mw_j\lvert \nabla(u^{n+1}(x_j)-u^n(x_j)) \rvert^2\\
&+\left(\frac{1}{\tau}+\frac12\right)\sum_{j=1}^mw_j(u^{n+1}(x_j)-u^n(x_j))^2 \\
 \leq& \frac32 \max\left\{\|u^n\|_{\infty}^2,\|u^{n+1}\|_{\infty}^2\right\}\sum_{j=1}^mw_j(u^{n+1}(x_j)-u^n(x_j))^2\\
 & + \left\langle f(u^n)-\mathcal{L}_N(f(u^n)),u^{n+1}-u^n\right\rangle_m.
\end{split}\end{equation*}
The property \eqref{equ:discreteprojection} suggests 
\begin{equation*}
\left\langle f(u^n)-\mathcal{L}_N(f(u^n)),u^{n+1}-u^n\right\rangle_m=0.
\end{equation*} 
Hence the estimate \eqref{equ:discreteenergyesti} holds.

When the quadrature exactness degree is $4N$, we know from \eqref{equ:energylem5}, \eqref{equ:energylem1}, and \eqref{equ:energylem2} that
\begin{equation}\label{equ:estilargem}\begin{split}
&\mathcal{E}(u^{n+1})-\mathcal{E}(u^n)+\frac{\nu^2}{2}\int_{\mathbb{S}^{d-1}}\lvert \nabla(u^{n+1}-u^n) \rvert^2{\rm{d}}\omega_d\\
&+\left(\frac{1}{\tau}+\frac12\right)\int_{\mathbb{S}^{d-1}}(u^{n+1}-u^n)^2 \text{d}\omega_d\\
 \leq& \frac32 \max\left\{\|u^n\|_{\infty}^2,\|u^{n+1}\|_{\infty}^2\right\}\int_{\mathbb{S}^{d-1}}(u^{n+1}-u^n)^2\text{d}\omega_d\\
 &+ \left\langle f(u^n)-\mathcal{L}_N(f(u^n)),u^{n+1}-u^n\right\rangle.
\end{split}\end{equation}
Note that 
\begin{equation*}
\left\langle f(u^n)-\mathcal{L}_N(f(u^n)),u^{n+1}-u^n\right\rangle= \left\langle f(u^n)-\mathcal{L}_N(f(u^n)),u^{n+1}-u^n\right\rangle_m,
\end{equation*}
because the quadrature exactness degree is $4N$. Thus, by the property \eqref{equ:discreteprojection} again, we have the estimate \eqref{equ:energyesti}.
\end{proof}

\begin{remark}
Lemma \ref{lem:energyesti} immediately suggests that if
\begin{equation*}
\frac{1}{\tau}+\frac12 \geq \frac32 \sup_{n\geq 0}\|u^n\|_{\infty}^2,
\end{equation*}
then $\tilde{\mathcal{E}}(u^{n+1})\leq \tilde{\mathcal{E}}(u^n)$ when the quadrature exactness degree is $2N$, and 
$\mathcal{E}(u^{n+1})\leq \mathcal{E}(u^n)$ when the quadrature exactness degree is $4N$.
\end{remark}

\begin{remark}
From the proof of Lemma \ref{lem:energyesti}, we can see that if we do not make the quadrature exactness assumption, the terms $\left\langle f(u^n)-\mathcal{L}_N(f(u^n)),u^{n+1}-u^n\right\rangle_m$ and $\left\langle f(u^n)-\mathcal{L}_N(f(u^n)),u^{n+1}-u^n\right\rangle$ cannot be guaranteed to be zero or negative. Thus, we cannot claim the (discrete) energy decay of the numerical solutions generated by \eqref{equ:scheme}. However, in practice, the estimate \eqref{equ:estilargem} may suggest that if we consider a sufficiently large number $m$ (depending on $\nu$) of quadrature points to construct $\mathcal{L}_N$ such that
\begin{equation*}
\frac{\nu^2}{2}\int_{\mathbb{S}^{d-1}}\lvert \nabla(u^{n+1}-u^n) \rvert^2{\rm{d}}\omega_d\geq \left\langle f(u^n)-\mathcal{L}_N(f(u^n)),u^{n+1}-u^n\right\rangle,
\end{equation*}
one may still have energy stability. We opt not to investigate this numerical issue, which is out of the scope of the paper.
\end{remark}

\begin{theorem}[Energy stability for $0<\tau\leq 1/2$]\label{thm:energystability1}
Let $0< \tau\leq 1/2$. Assume $u_0\in H^{s}(\mathbb{S}^{d-1})$ with $s>d-1$ and $\|u_0\|_{\infty}\leq 1$. Then, there exists $N_5:=N_5(\nu,s,d,u_0)$ such that, for $N\geq N_5$, we have the discrete energy decay
\begin{equation*}
\tilde{\mathcal{E}}(u^{n+1})\leq\tilde{\mathcal{E}}(u^{n}),\quad n\geq 0
\end{equation*}
if the quadrature rule \eqref{equ:quad} has exactness degree $2N$, and the energy decay
\begin{equation*}
\mathcal{E}(u^{n+1})\leq\mathcal{E}(u^{n}),\quad n\geq 0
\end{equation*}
if the quadrature rule \eqref{equ:quad} has exactness degree $4N$.
\end{theorem}

\begin{proof}
Let $\alpha_0=\sqrt{5/3}-1$ in Corollary \ref{cor:quadrature}. Then, there exists $N_5(\nu,s,d,u_0)$ such that, for $N\geq N_5$, it holds that
\begin{equation*}
\sup_{n\geq 0 }\|u^n \|_{\infty}\leq\sqrt{\frac{5}{3}}.
\end{equation*}
Thus, we have
\begin{equation*}
\frac{1}{\tau}+\frac12\geq \frac52\geq \frac32\sup_{n\geq 0 }\|u^n \|_{\infty}^2,
\end{equation*}
and furthermore, both energy decaying estimates follow directly from Lemma \ref{lem:energyesti}.
\end{proof}

With the aid of Theorem \ref{thm:uniformstability2} and Remark \ref{rem:dongli}, we now derive the energy stability result for $\tau\geq1/2$. This result is only valid for $1/2<\tau<\tau_1\approx 0.86$. Consider the equation
\begin{equation*}
\frac12+\frac1x=\frac32\cdot\left(\frac23\cdot \frac{(1+x)^{3/2}}{\sqrt{3x}}\right)^2.
\end{equation*}
It is easy to check that
\begin{equation*}
x=\tau_1=\frac12\left(-2+(9-3\sqrt{6})^{1/3}+(9+3\sqrt{6})^{1/3}\right)\approx 0.860018
\end{equation*}
is the unique real-valued solution to this equation. Thus, if $1/2<\tau\leq\tau_1-\epsilon_0$, where $0<\epsilon_0\leq 0.1$, then
\begin{equation}\label{equ:etaepsilon0}
\frac12+\frac{1}{\tau}\geq\frac32\left(\frac{(1+\tau)^{3/2}}{\sqrt{3\tau}}\cdot\frac23+\zeta(\epsilon_0)\right)^2,
\end{equation}
where $\zeta(\epsilon_0)>0$ only depends on $\epsilon_0$. Thus, we have the following theorem.

\begin{theorem}[Energy stability for $1/2< \tau<\tau_1$]\label{thm:energystability2}
Let $1/2<\tau\leq \tau_1-\epsilon_0$ for some $0<\epsilon_0\leq 0.1$, and let
\begin{equation*}
M_1 = \frac{(1+\tau)^{3/2}}{\sqrt{3\tau}}\cdot\frac23+\zeta(\epsilon_0),
\end{equation*}
where $\zeta(\epsilon_0)$ is the same as the one in \eqref{equ:etaepsilon0}. Assume $u_0\in H^{s}(\mathbb{S}^{d-1})$ with $s>d-1$ and $\|u^0\|_{\infty}\leq M_1$. If $N\geq N_6:=N_6(\tau,\epsilon_0,\nu,s,d,u_0)$, then we have the discrete energy decay
\begin{equation*}
\tilde{\mathcal{E}}(u^{n+1})\leq\tilde{\mathcal{E}}(u^{n}),\quad n\geq 0
\end{equation*}
if the quadrature rule \eqref{equ:quad} has exactness degree $2N$, and the energy decay
\begin{equation*}
\mathcal{E}(u^{n+1})\leq\mathcal{E}(u^{n}),\quad n\geq 0
\end{equation*}
if the quadrature rule \eqref{equ:quad} has exactness degree $4N$.
\end{theorem}
\begin{proof}
With $\eta = 0$, Theorem \ref{thm:uniformstability2} and Remark \ref{rem:dongli} immediately suggest the $L^{\infty}$ stability of
\begin{equation*}
\sup_{n\geq 0}\|u^n\|_{\infty}\leq M_1.
\end{equation*}
In light of the energy estimates in Lemma \ref{lem:energyesti}, it suffices to ensure
\begin{equation*}
\frac12+\frac{1}{\tau}\geq\frac32 M_1^2=\frac32\left(\frac{(1+\tau)^{3/2}}{\sqrt{3\tau}}\cdot\frac23+\eta(\epsilon_0)\right)^2,
\end{equation*}
which is exactly \eqref{equ:etaepsilon0}.
\end{proof}

\subsection{An mixed quadrature-based scheme}
Theoretical results in Section \ref{sec:theory} suggest that the new scheme \eqref{equ:scheme} may not have energy stability if the quadrature exactness \eqref{equ:exactness} is not assumed. Recall that our third motivation for studying the scheme \eqref{equ:scheme} is that it may not be practical to acquire desirable samples of the initial condition from quadrature points. Regarding this potential limitation, we can consider the following mixed quadrature-based scheme
\begin{equation}\label{equ:enhancedscheme}
\begin{cases}
&\dfrac{u^{n+1}-u^n}{\tau}=\nu^2\Delta u^{n+1}-\tilde{\mathcal{L}}_{N}\left((u^n)^3-u^n\right),\quad n\geq 0,\vspace{0.15cm}\\
& u^0 = {\mathcal{L}}_{N}u_0,
\end{cases}
\end{equation}
where $\mathcal{L}_N$ is constructed by quadrature rules \eqref{equ:quad} satisfying Assumption \ref{assumption:1MZ} only and $\tilde{\mathcal{L}}_N$ is the hyperinterpolation operator constructed by quadrature rules with quadrature exactness degree of $2N$ or $4N$. Thus, if $u_0\in H^{s}(\mathbb{S}^{d-1})$ with $s>d-1$, $s_0>\frac{d-1}{2}$ and $\eta= \tilde{c} N^{-\varepsilon}$ for any $\tilde{c}\geq0$ and $\varepsilon >s_0$. Thus, the performance of the mixed quadrature-based scheme \eqref{equ:enhancedscheme} can also be characterized by Corollary \ref{cor:quadrature}, Theorem \ref{thm:energystability1}, and Theorem \ref{thm:uniformstability2}. The imposed assumptions only aim to guarantee \eqref{equ:positivealpha}. Thus, with this mixed quadrature-based scheme, even for a set of scattered data of $u_0$, it is still possible to generate a sequence of numerical solutions quantified by Corollary \ref{cor:quadrature}.

\section{Numerical experiments}\label{sec:example}

In this section, we present some numerical experiments on the 2-sphere $\mathbb{S}^2\subset\mathbb{R}^3$ to verify the theoretical assertions presented in the previous sections. It is worth noting that $|\mathbb{S}^2| = 4\pi$. For simplicity, we consider quadrature rules \eqref{equ:quad} with equal-weight weights
\begin{equation*}
w_j={4\pi}/{m},\quad j=1,2,\ldots,m.
\end{equation*}
Numerous point sets on the sphere have been introduced in the literature. In our experiments, we use the following points sets: 1) randomly scattered points generated in MATLAB; 2) equal area points \cite{MR1306011} generated based on an algorithm given in \cite{MR2582801}; 3) Fekete points which maximize the determinant for polynomial interpolation \cite{MR2065291}; 4) Coulomb energy points which minimize $\sum_{i,j=1}^m(1/\|x_i-x_j\|_2)$ \cite{MR1845243}; and 5) well-conditioned spherical $t$-designs proposed in \cite{MR2763659}. Fekete points and Coulomb energy points are precomputed by R. Womersley and are available on his website\footnote{Robert Womersley, \emph{Interpolation and Cubature on the Sphere}, \url{http://
www.maths.unsw.edu.au/~rsw/Sphere/}; accessed in March, 2023.}. All codes were written by MATLAB R2022a, and all numerical experiments were conducted on a laptop (16 GB RAM, Intel CoreTM i7-9750H Processor) with macOS Monterey 12.5.

 We begin with an experiment to illustrate how the phases are separated using the above-mentioned five different types of quadrature points. We set $\nu=10^{-1}$, $\tau=0.5$, and $N = 15$, consider the initial condition
 \begin{equation}\label{equ:testinitial}
 u(0, x, y, z) = \cos(\text{cosh}(5xz) - 10y),
 \end{equation}
and solve for $u$ up to time $t = 70$. The numerical solutions at times $t= 0,5,10,15,70$ are shown in Figure \ref{fig:Ex1}. The initial condition quickly converges to a metastable state $u\approx \pm 1$ (yellow area indicates $u\approx1$, and blue area indicates $u\approx-1$) at time around $t=10$ (for equal area points, Coulomb energy points, and spherical $t$-designs) and around $t=15$ (for random points and Fekete points), and it eventually reaches the stable state $u=1$ at around $t=70$. We note that random points may perform slightly worse than points with certain properties, and the inferior performance of Fekete points may be due to the fact that all computed Fekete points are only approximate local maximizers of the determinant for polynomial interpolation. Nevertheless, this experiment demonstrates how our proposed practical scheme \eqref{equ:scheme} works.

\begin{figure}[h]
  \centering
  \begin{subfigure}
  \centering
  \includegraphics[width=11.5cm]{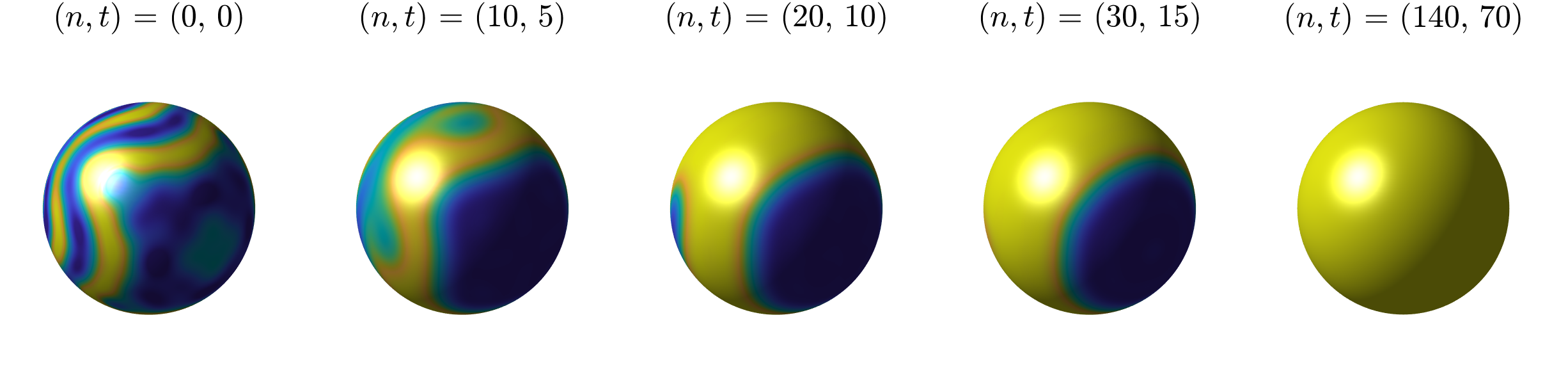}
\end{subfigure}\\
\begin{subfigure}
  \centering
  \includegraphics[width=11.5cm]{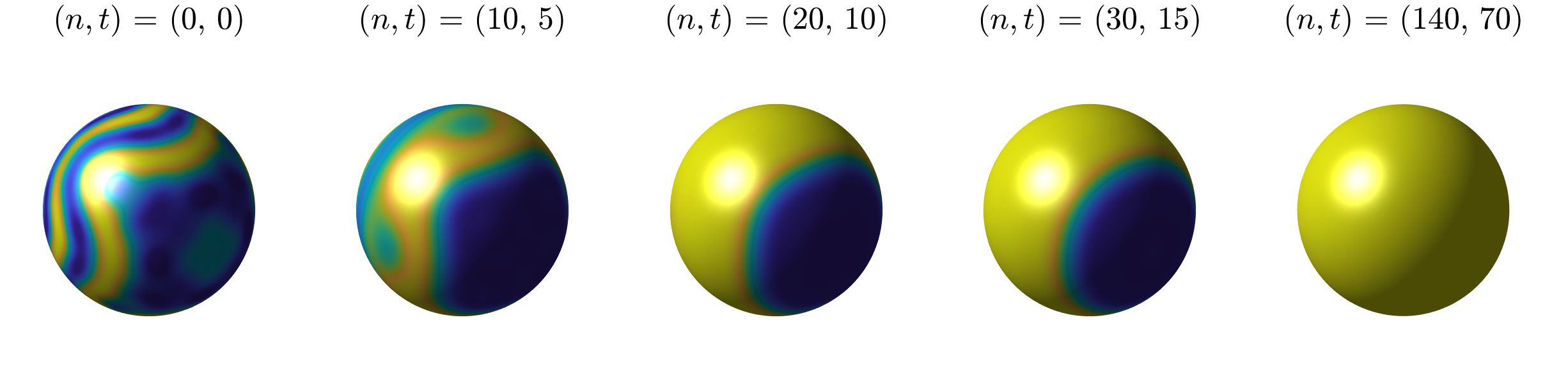}
\end{subfigure}\\
\begin{subfigure}
  \centering
  \includegraphics[width=11.5cm]{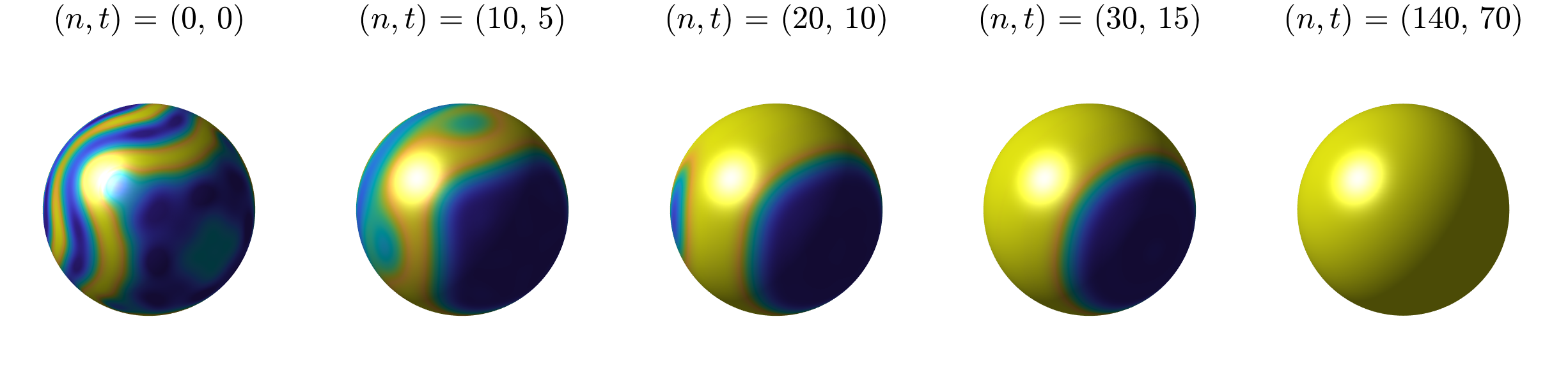}
\end{subfigure}\\
\begin{subfigure}
  \centering
  \includegraphics[width=11.5cm]{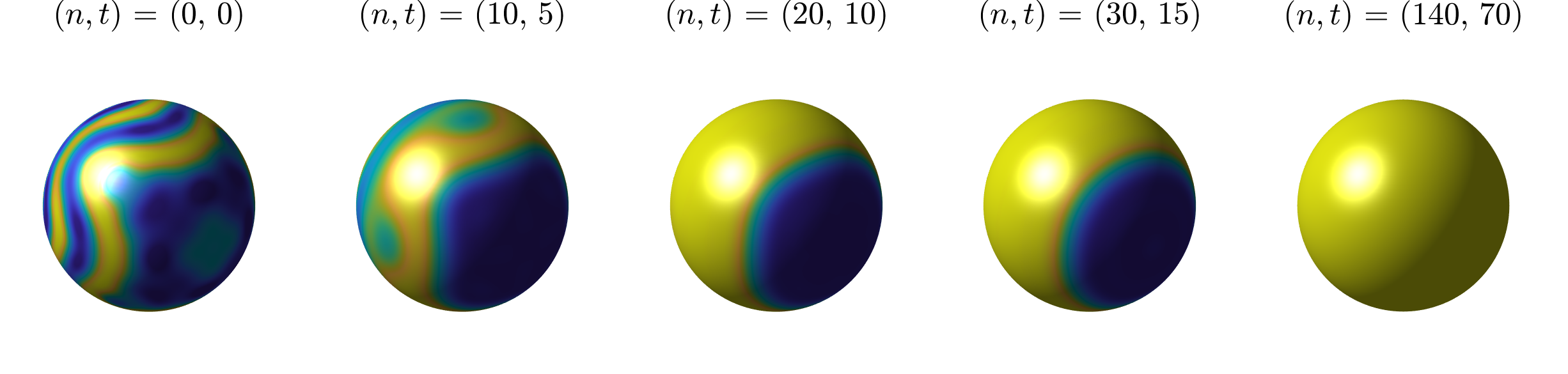}
\end{subfigure}\\
\begin{subfigure}
  \centering
  \includegraphics[width=11.5cm]{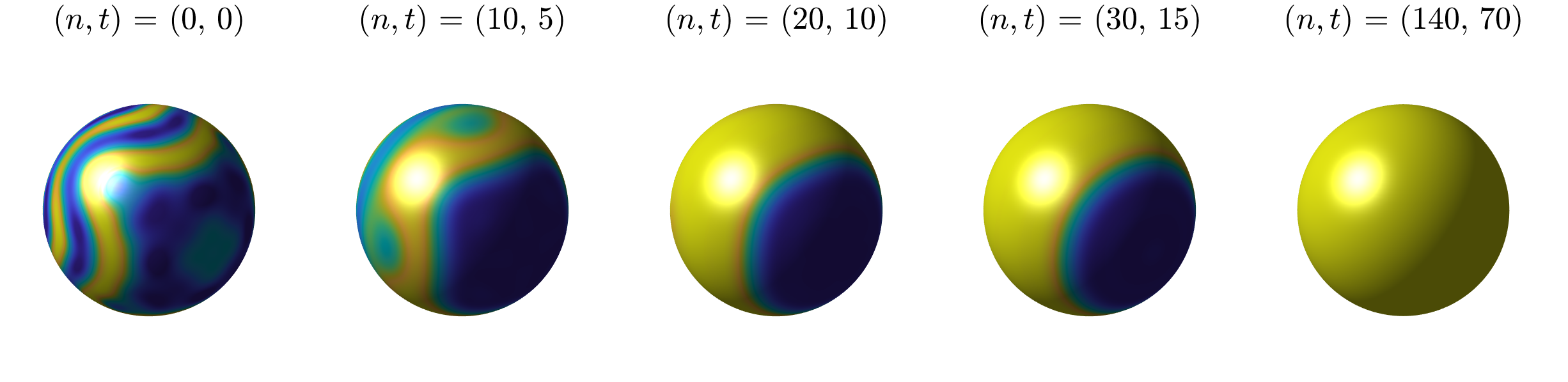}
\end{subfigure}\\
  \caption{Numerical solution to the Allen--Cahn equation \eqref{equ:AC} with $\nu=0.1$ and initial condition \eqref{equ:testinitial} using our scheme \eqref{equ:scheme} with $\tau=0.5$, $N = 15$, and different quadrature points. From top row to bottom row: $m = \lfloor 120N^2\ln{N}\rfloor =73,117$ random points; $m = (2N+1)^2 =961$ equal area points; $m =961$ Fekete points; $m =961$ Coulomb energy points; and $m =961$ spherical $2N$-designs. }\label{fig:Ex1}
\end{figure}

In our second experiment, we aim to test the effective maximum principle and the $L^{\infty}$ stability of the numerical solutions generated by our scheme \eqref{equ:scheme}, without the quadrature exactness assumption \eqref{equ:exactness}. Namely, we verify the theoretical assertions in Section \ref{sec:theory} using random points, equal area points, Fekete points, and Coulomb energy points. The uniform norms $\|u^n\|_{\infty}$ of the numerical solution $u^n$ to the Allen--Cahn equation \eqref{equ:AC} with $\nu=0.1$ and initial condition \eqref{equ:testinitial} are documented in Figure \ref{fig:Ex2}, in which we set $\tau\in\{0.5,1,1.99\}$, $N \in\{10, 16, 24\}$, and $m =\lfloor 120N^2\ln{N}\rfloor $ for random points and $m = (2N+1)^2$ for equal area points, Fekete points, and Coulomb energy points.

We theoretically demonstrate that, if $\tau\leq 1/2$, then the effective maximum principle holds. That is, for any fixed $N$, the upper bound of $\|u^n\|_{\infty}$ decreases as time advances. This principle suggests that,  although $\|u^{n}\|_{\infty}$ may backtrack, it eventually decreases. This is verified by the first column of Figure \ref{fig:Ex2}, in which $\tau=0.5$ ensures the effective maximum principle. If $1/2<\tau<2$, then from the $L^{\infty}$ stability result, we know that $\|u^n\|_{\infty}$ is bounded by $\|u^0\|_{\infty}$, which is illustrated by the second and third columns of Figure \ref{fig:Ex2}.

\begin{figure}[htbp]
  \centering
  \begin{subfigure}
  \centering
  \includegraphics[width=\textwidth]{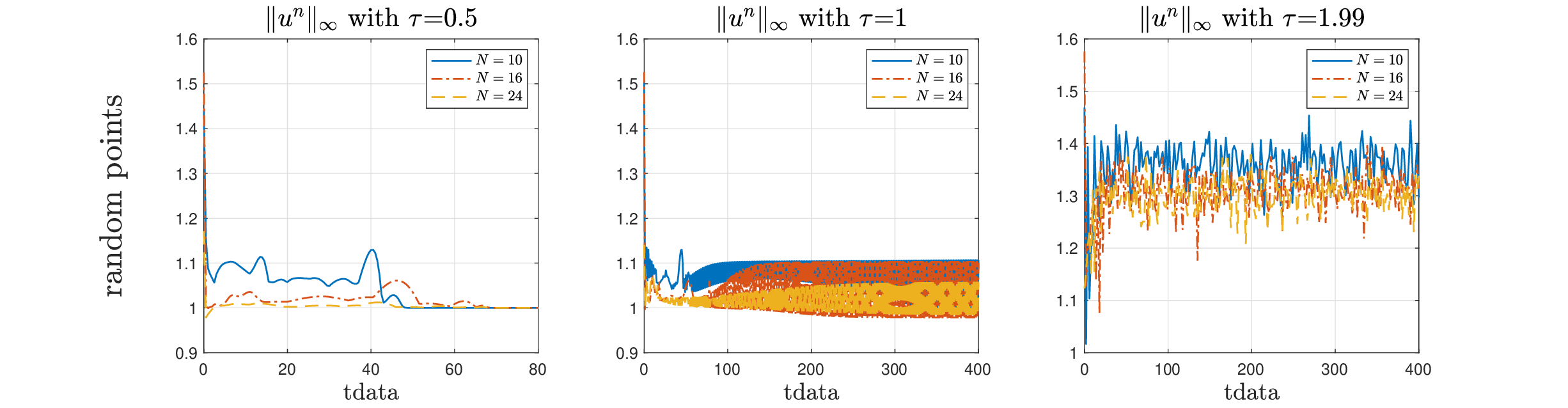}
\end{subfigure}\\
\begin{subfigure}
  \centering
  \includegraphics[width=\textwidth]{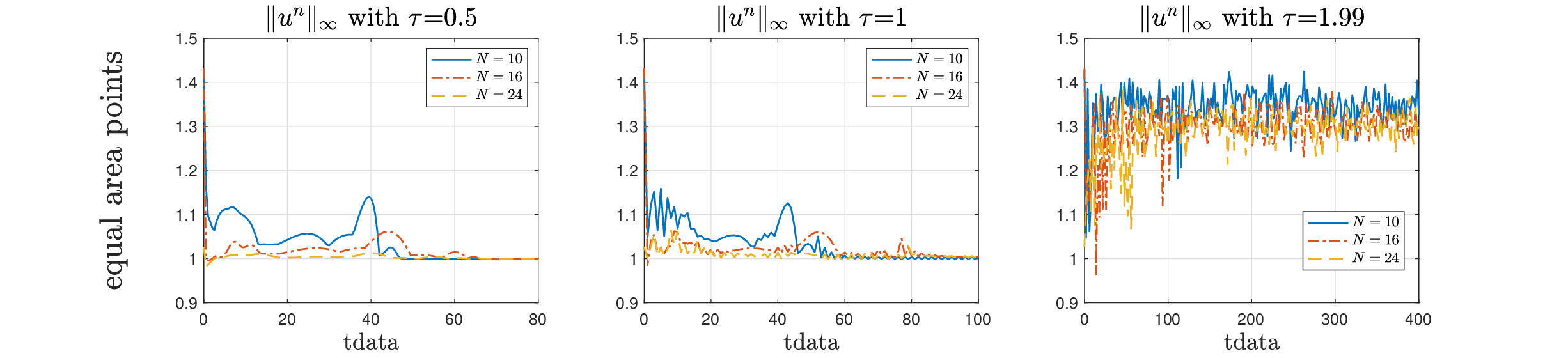}
\end{subfigure}\\
\begin{subfigure}
  \centering
  \includegraphics[width=\textwidth]{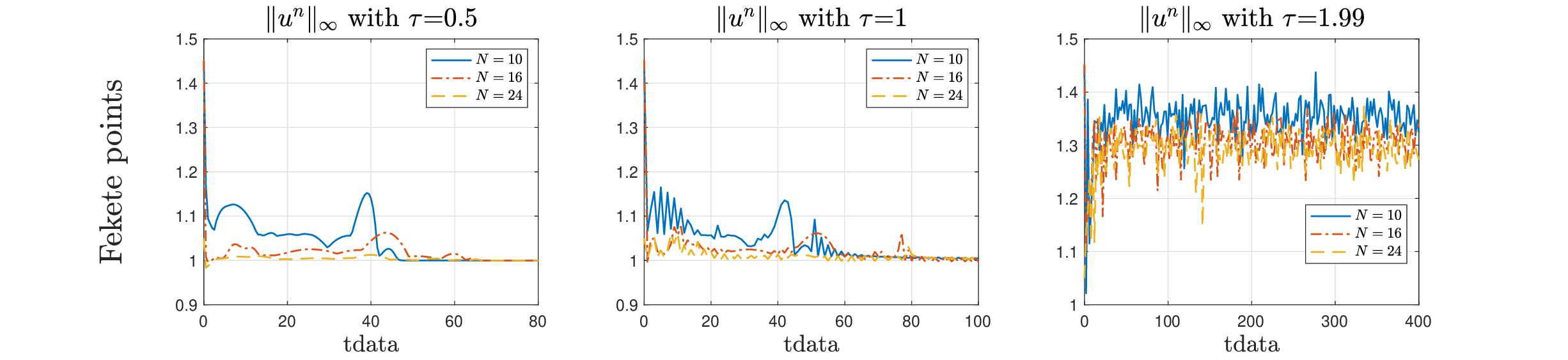}
\end{subfigure}\\
\begin{subfigure}
  \centering
  \includegraphics[width=\textwidth]{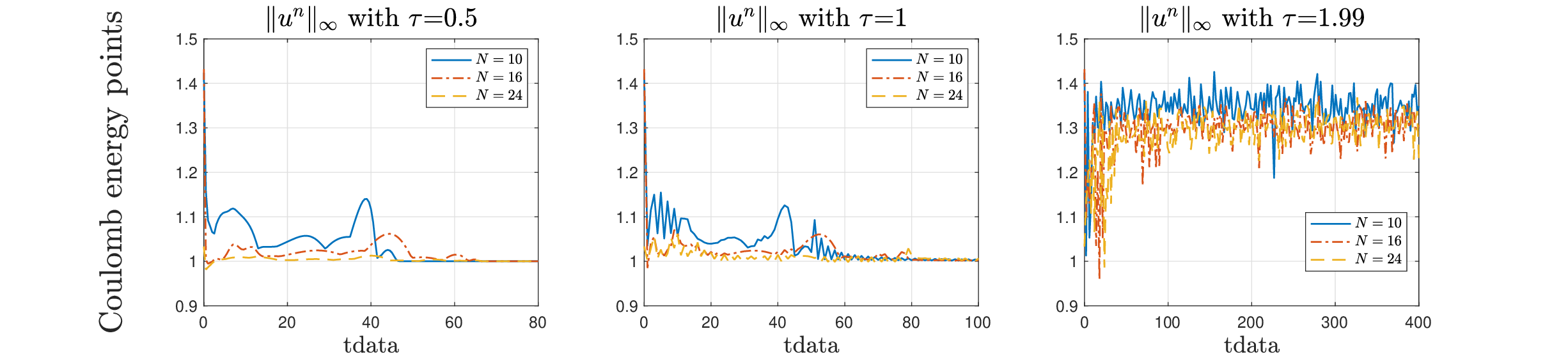}
\end{subfigure}\\
  \caption{Uniform norms of the numerical solution to the Allen--Cahn equation \eqref{equ:AC} with $\nu=0.1$ and initial condition \eqref{equ:testinitial} using our scheme \eqref{equ:scheme} with $\tau\in\{0.5,1,1.99\}$, $N \in\{10, 16, 24\}$, and $m =\lfloor 120N^2\ln{N}\rfloor $ for random points and $m = (2N+1)^2$ for equal area points, Fekete points, and Coulomb energy points.}\label{fig:Ex2}
\end{figure}

In our third experiment, we investigate the energy decay of the new scheme \eqref{equ:scheme} and test the mixed quadrature-based scheme \eqref{equ:enhancedscheme} discussed in Section \ref{sec:discussion}. Since our analysis in Section \ref{sec:discussion} relies on quadrature exactness, we consider spherical $t$-designs. Recall that, when $0<\tau\leq 0.86$, the scheme \eqref{equ:scheme} using quadrature rules of exactness degree $2N$ ensures discrete energy decay $\tilde{\mathcal{E}}(u^{n+1})\leq\tilde{\mathcal{E}}(u^{n})$ for $n\geq 0$, if the degree $N$ is sufficiently large. Moreover, for a sufficiently large $N$, the scheme \eqref{equ:scheme} has energy decay
$\mathcal{E}(u^{n+1})\leq\mathcal{E}(u^{n})$ for $n\geq 0$ if the quadrature exactness degree is $4N$. The energy profiles of the numerical solution $u^n$ to the Allen--Cahn equation \eqref{equ:AC} with $\nu=0.1$ and initial condition \eqref{equ:testinitial} are illustrated in Figure \ref{fig:Ex3}, in which we set $\tau\in\{0.1,0.5,0.86\}$ and $N \in \{12, 14, 16\}$. Despite that the energy dissipation property holds for all cases, it seems that the time stepping size significantly influences the energy evolution.

\begin{figure}[htbp]
  \centering
  \includegraphics[width=\textwidth]{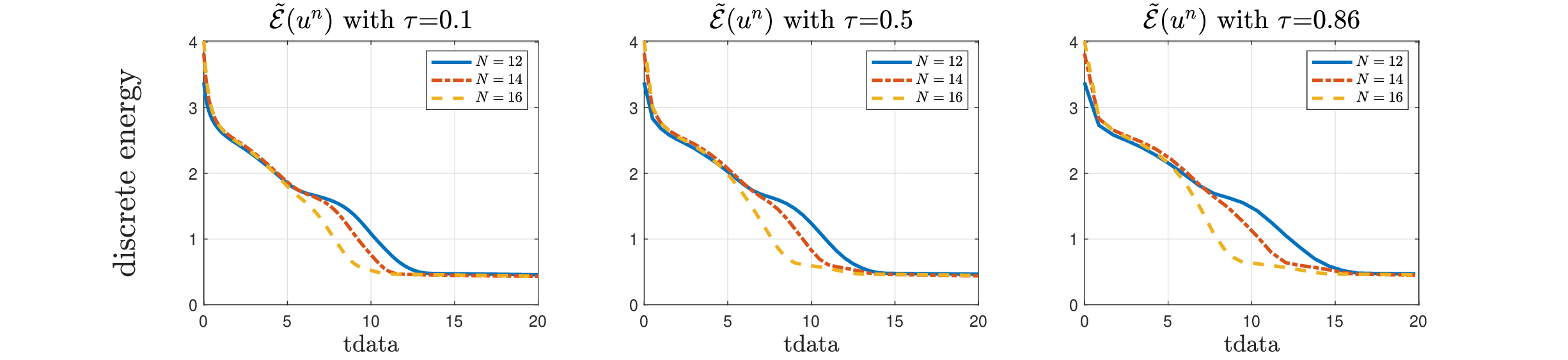}  \\
  \includegraphics[width=\textwidth]{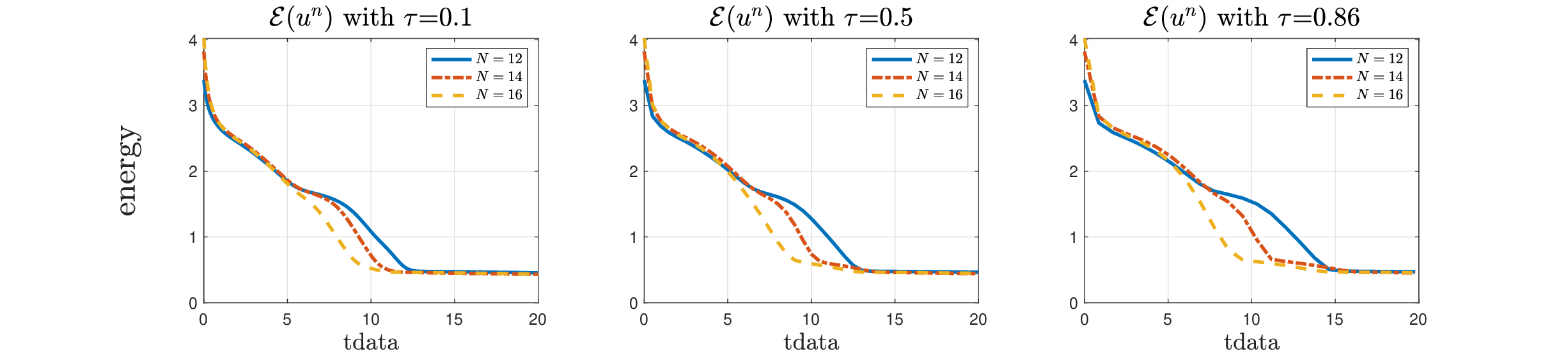}  \\
  \caption{Energy profiles of the numerical solution to the Allen--Cahn equation \eqref{equ:AC} with $\nu=0.1$ and initial condition \eqref{equ:testinitial} using our scheme \eqref{equ:scheme} with $\tau\in\{0.1,0.5,0.86\}$ and $N \in \{12, 14, 16\}$. Top row: using spherical $2N$-designs; Bottom row: using spherical $4N$-designs.}\label{fig:Ex3}
\end{figure}

It is worth noting that quadrature exactness of degree at least $2N$ is necessary for energy dissipation, as evidenced by the following counterexample. Figure \ref{fig:Ex4} records the energy evolution of the numerical solution to the Allen--Cahn equation \eqref{equ:AC} with $\nu=0.01$, and initial condition \eqref{equ:testinitial} using the enew scheme \eqref{equ:scheme} with $\tau=0.86$, and different values of $N$. If the quadrature exactness is only of degree $N$, as shown in the top row of Figure \ref{fig:Ex4}, the discrete energy $\tilde{E}(u^n)$ fails to dissipate, and increasing $N$ does not resolve this issue. On the other hand, if the quadrature exactness degree is $2N$, our refined analysis in Section \ref{sec:discussion} guarantees that discrete energy dissipation always occurs, as demonstrated by the middle row of Figure \ref{fig:Ex4}. Furthermore, if we consider the mixed quadrature-based scheme \eqref{equ:enhancedscheme} proposed in Section \ref{sec:discussion}, where the hyperinterpolation operator with quadrature exactness $N$ is used for projecting $u_0$ to $u^0$ and another hyperinterpolation operator with quadrature exactness $2N$ is used in time evolution, then solutions generated by this scheme exhibit energy dissipation. This is well  shown in the bottom row of Figure \ref{fig:Ex4}.

\begin{figure}[htbp]
  \centering
  \includegraphics[width=\textwidth]{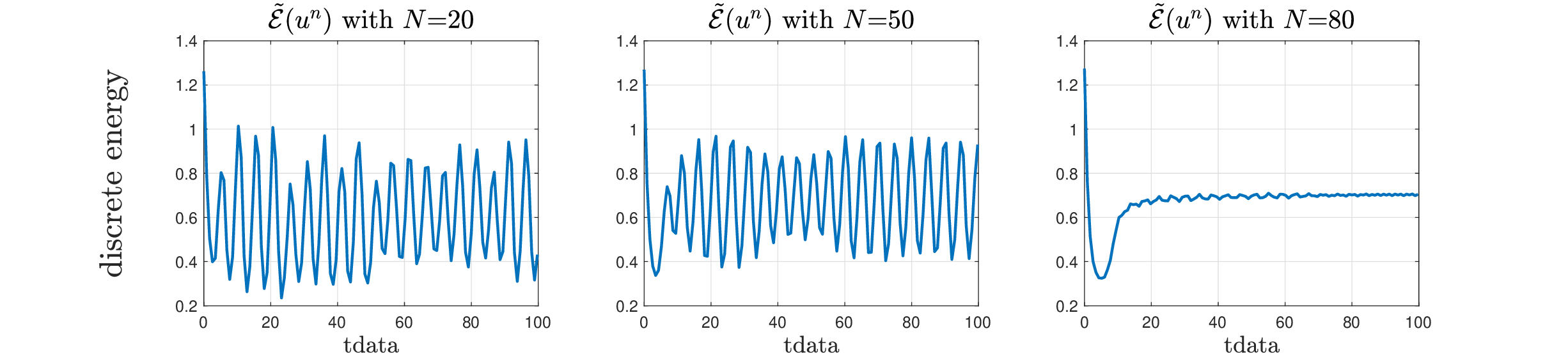}  \\
  \includegraphics[width=\textwidth]{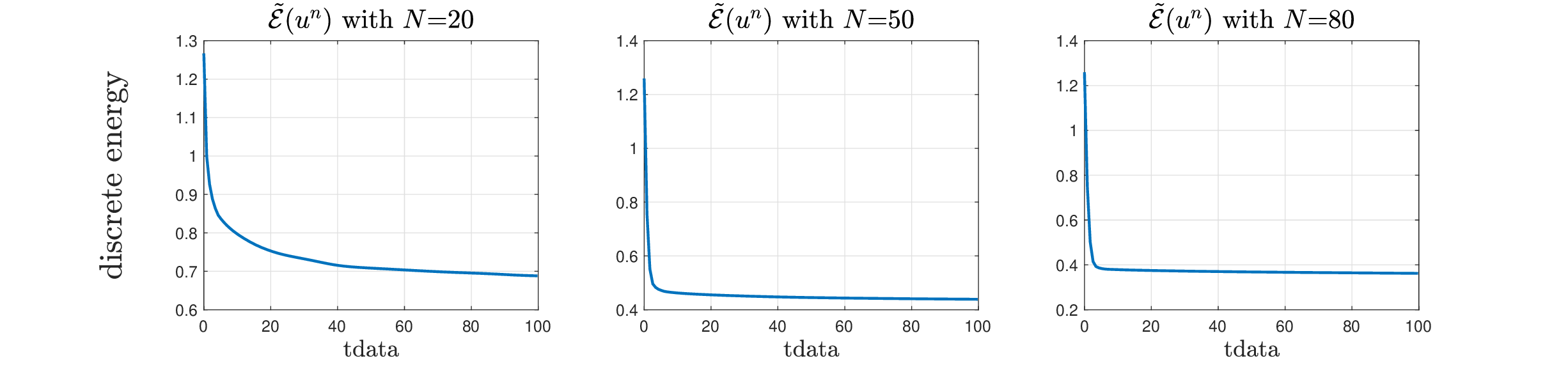}  \\
  \includegraphics[width=\textwidth]{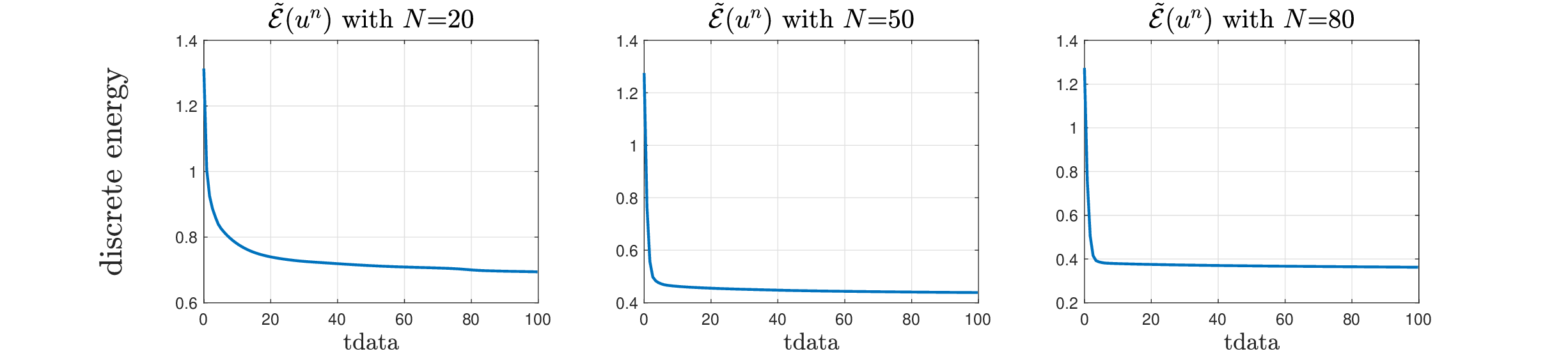}  \\
  \caption{Energy profiles of the numerical solution to the Allen--Cahn equation \eqref{equ:AC} with $\nu=0.01$ and initial condition \eqref{equ:testinitial} using our scheme \eqref{equ:scheme} with $\tau=0.86$ and $N \in\{20,50,80\}$. Top row: quadrature exactness of degree $N$; Middle row: quadrature exactness of degree $2N$; Bottom row: the mixed quadrature-based scheme \eqref{equ:enhancedscheme}.}\label{fig:Ex4}
\end{figure}

\section{Conclusions and discussions}\label{sec:conclusion}

We proposed a novel quadrature-based spectral method for solving the Allen--Cahn equation on spheres. For the new method, we studied its $L^{\infty}$ stability, energy stability, and the effective maximum principle proposed recently by Li in \cite{MR4294331}, based only on assumptions regarding the degree $N$ of the polynomial numerical solutions. These theoretical results differ from those in the Allen--Cahn literature in the sense that there is no need to assume any Lipschitz property of the nonlinear term in the Allen-Cahn equation or any \emph{a priori} $L^{\infty}$ boundedness of the numerical solutions. Moreover, the new method is proved to be suitable for long-time simulations because the time stepping size is independent of the diffusion coefficient $\nu$ in the equation. We summarize our main theoretical results in Table \ref{tab:map}. Our discussion also addresses the specific scenario where desirable data samples may be unavailable at certain quadrature points. This can occur when full access to the initial condition $u_0$ is restricted, and only a set of samples without prescribed locations is available. From the perspective of numerical analysis, our analysis also aligns with the recent trend in re-accessing the necessity of quadrature exactness, because what matters in practice is the accuracy for integrating non-polynomial functions.

\begin{table}[htbp]
\centering
\footnotesize \caption{The map of theoretical results}\label{tab:map}
\begin{tabular}{|c|c|c|}
\hline
Types of results      & $0<\tau\leq1/2$ & $1/2<\tau<2$\\\hline
$L^{\infty}$ stability & Theorem \ref{thm:uniformstability1}    &   Theorem \ref{thm:uniformstability2} \\\hline
Energy stability (with quadrature exactness) & Theorem \ref{thm:energystability1}  &   Theorem \ref{thm:energystability2} (only for $\tau\leq0.86$)\\\hline
Effective maximum principle & Theorem \ref{thm:effectivemaxprinciple1} & Nil  \\\hline
\end{tabular}
\end{table} 

Our approach discretizes the Allen-Cahn equation using intrinsic spherical coordinates. This differs fundamentally from methods that first parameterize the sphere in spherical coordinates then apply Euclidean-based discretization - an approach that inevitably encounters coordinate singularities at the poles (where the transformation's Jacobian becomes singular). By working entirely within the sphere's intrinsic geometry, we avoid these numerical issues while preserving the domain's geometric properties. The method extends naturally to any closed surface diffeomorphic to a sphere, see, e.g., the manipulation of the change of variables in \cite{graham2002fully}.

It is interesting to consider extending our approach to other semi-linear PDEs \eqref{equ:PDE} on other domains, with the nonlinear part $\mathbf{N}(u)$ linearized by its hyperinterpolant. Our implementation framework can be readily extended to any domain where hyperinterpolation is applicable or can be constructed, provided two key components are available: an orthonormal basis and a quadrature rule for integration. It should be noted that we leverage a key property of our basis functions that the spherical harmonics are eigenfunctions of the negative Laplace-Beltrami operator with explicit eigenvalues, eliminating the need for additional spatial discretization of differential operators on spheres. However, on domains other than tori or spheres, it may be necessary to discretize differential operators in an additional step involving the computation of basis polynomial derivatives. For theoretical analysis of the extended schemes, it may follow the analysis in Sections \ref{sec:theory} and \ref{sec:discussion}, with slight modifications on the definitions of corresponding PDEs. Our analysis relies on the Marcinkiewicz-Zygmund property \eqref{equ:MZproperty}. While this property has been well-established for compact manifolds \cite{filbir2011marcinkiewicz} and various Euclidean domains \cite{MR3746524}, its verification remains an essential step when applying our framework to particular domains. Nonetheless, the new method may be promising for numerically solving a wide range of semi-linear PDEs in bounded and closed regions of $\mathbb{R}^d$, where hyperinterpolation can be defined.

~\\
\noindent \textbf{Acknowledgements.} The work of X. Yuan was supported by the Hong
Kong Research Grants Council under the GRF project 17309824.
\bibliographystyle{siam}
\bibliography{myref}

\clearpage

\end{document}